\documentclass[12pt, reqno]{amsart}
\usepackage{amsmath, amsthm, amscd, amsfonts, amssymb}
\usepackage[mathscr]{eucal}

\usepackage[bookmarksnumbered, colorlinks, plainpages]{hyperref} 

\theoremstyle{plain}
\newtheorem{theorem}{Theorem}

\newtheorem{proposition}[theorem]{Proposition}
\newtheorem{corollary}[theorem]{Corollary}

\theoremstyle{remark}

\newtheorem{example}{Example}

\theoremstyle{definition}
\newtheorem{definition}[theorem]{Definition}

\DeclareMathAlphabet\mathoo{U}{eur}{b}{n}
\DeclareMathOperator*{\esssup}{ess\,sup}

\begin{document}
\setcounter{page}{1}
\title[Locally nuclear operators]{Inverse-closedness of the subalgebra\\ of locally nuclear operators}

\author[E.~Yu. Guseva]{E.~Yu. Guseva}
 \address{Department of System Analysis and Control,
Voronezh State University\\ 1, Universi\-tet\-skaya Square, Voronezh 394018, Russia}
\email{elena.guseva.01.06@gmail.com}

\author{V.~G. Kurbatov}
 \address{Department of System Analysis and Control,
Voronezh State University\\ 1, Universi\-tet\-skaya Square, Voronezh 394018, Russia}
\email{kv51@inbox.ru}



\date{\today}

\begin{abstract}
Let $X$ be a Banach space and $T$ be a bounded linear operator acting in $l_p(\mathbb Z^c,X)$, $1\le p\le\infty$. The operator $T$ is called \emph{locally nuclear} if it can be represented in the form
\begin{equation*}
(Tx)_k=\sum\limits_{m\in\mathbb Z^c} b_{km}x_{k-m},\qquad k\in\mathbb Z^c,
\end{equation*}
where $b_{km}:\,X\to X$ are nuclear,
\begin{equation*}
\lVert b_{km}\rVert_{\mathfrak S_1}\le\beta_{m},\qquad k,m\in\mathbb Z^c,
\end{equation*}
$\lVert\cdot\rVert_{\mathfrak S_1}$ is the nuclear norm,
$\beta\in l_{1}(\mathbb Z^c,\mathbb C)$ or $\beta\in l_{1,g}(\mathbb Z^c,\mathbb C)$, and $g$ is an appropriate weight on $\mathbb Z^c$. It is established that if $T$ is locally nuclear and the operator $\mathbf1+T$ is invertible, then the inverse operator $(\mathbf1+T)^{-1}$ has the form $\mathbf1+T_1$, where $T_1$ is also locally nuclear.
This result is refined for the case of operators acting in $L_p(\mathbb R^c,\mathbb C)$.
\end{abstract}

\maketitle

\section*{Introduction}\label{s:Introduction}
A bounded linear operator $A:\,X\to X$, where $X$ is a Banach space, is called~\cite{Grothendieck66,Pietch78:eng,Ruston51a,Ruston51b} \emph{nuclear} if it can be represented in the form
\begin{equation}\label{e:repr of A:0}
Ax=\sum_{i=1}^\infty a_i(x)y_i,
\end{equation}
where $y_i\in X$, $a_i\in X^*$ (here $X^*$ is the conjugate of $X$), and
\begin{equation*}
\sum_{i=1}^\infty\lVert a_i\rVert\cdot\lVert y_i\rVert<\infty.
\end{equation*}
The space $\mathfrak S_1=\mathfrak S_1(X)$ of nuclear operators is Banach with respect to the norm
\begin{equation}\label{e:nuclear norm}
\lVert A\rVert_{\mathfrak S_1}=\inf\sum_{i=1}^\infty\lVert a_i\rVert\cdot\lVert y_i\rVert,
\end{equation}
where the infimum is taken over all representations of the operator $A$ in the form~\eqref{e:repr of A:0}.

The main part of applications of operator theory to numerical mathematics~\cite{Golub-Van_Loan13:eng} is associated with finite-dimensional operators, i.~e. having a finite-dimensional image, although, possibly, acting in an infinite-dimensional space. The space $\mathfrak S_1$ of nuclear operators forms a small extension of the space of finite-dimensional operators. Nuclear operators are more convenient from the theoretical point of view, because, in contrast to finite-dimensional ones, they form a Banach space (with respect to norm~\eqref{e:nuclear norm}).

A subclass $\mathfrak T$ of the algebra of all bounded linear operators acting in a Banach space is called a \emph{subalgebra} if it is closed under addition, multiplication by scalars, and composition. If additionally $\mathfrak T$ is closed under inversion, the subalgebra $\mathfrak T$ is called \emph{full} or \emph{inverse-closed}. Many classes $\mathfrak T$ of compact operators (after adjoining the identity operator) form full subalgebras, see, e.~g.,~\cite[Theorem 2.3]{Zabrejko68:eng-rus}. The inverse-closedness of subalgebras consisting of compact operators is usually intimately related to the fact that the resolvent set of a compact operator is connected. For example, nuclear operators are compact (Proposition~\ref{p:Pietch 1.11.2}), and the subalgebra $\mathfrak S_1$ of nuclear operators is full (Theorem~\ref{t:nuclear is full}). If all operators involved in a mathematical problem belong to the same full subalgebra, then the solution of the problem usually also belongs to the same subalgebra. Thus, we obtain some qualitative information about the solution in advance, which can simplify the investigation of the problem. It is clear that the narrower the full subalgebra, the more simple and more convenient it is to work with it.

Unfortunately, operators involved in some applications are substantially non-compact. Examples of such applications are the theory of stability~\cite{Antoulas}, the theory of pseudo-differential operators~\cite{Taylor-PDE81}, and many others. Nevertheless, even in such applications, there naturally arise some classes of operators close to compact ones. This paper is devoted to two such classes $\mathoo s_{1,g}$ and $\mathoo S_{1,g}$. We call operators belonging to $\mathoo s_{1,g}$ and $\mathoo S_{1,g}$ \emph{locally nuclear}. Roughly, an operator acting in $L_p(\mathbb R^c,\mathbb C)$ is locally nuclear if its restriction to any compact subset $M\subset\mathbb R^c$ is nuclear and its memory decreases at infinity in a special way. Our main results (Theorems~\ref{t:fin:convolution dominayed:nuclear} and~\ref{t:convolution dominayed:nuclear:2}) state that locally nuclear operators form full subalgebras. We also show that a locally nuclear operator acting in the space $L_p(\mathbb R^c,\mathbb C)$, $1\le p<\infty$, admits an integral representation (Theorem~\ref{t:loc nuclear oper in L_p}).

A more precise definition of a locally nuclear operator is as follows. Let a linear operator $T$ act in $l_p(\mathbb Z^c,X)$, $1\le p\le\infty$, where $X$ is a Banach space. We say that the operator $T$ belongs to the class $\mathoo s_{1,g}$ if it can be represented in the form
\begin{equation}\label{e:operator D:0}
(Tx)_k=\sum\limits_{m\in\mathbb Z^c} b_{km}x_{k-m},\qquad k\in\mathbb Z^c,
\end{equation}
where $b_{km}:\,X\to X$ are nuclear and
\begin{equation}\label{e:beta:0}
\lVert b_{km}\rVert_{\mathfrak S_1}\le\beta_{m}
\end{equation}
for an appropriate (see Definition~\ref{def:weight}) function $\beta\in l_{1}(\mathbb Z^c,\mathbb C)$ or $\beta\in l_{1,g}(\mathbb Z^c,\mathbb C)$, where $g$ is a weight on $\mathbb Z^c$.
Next, let a linear operator $A$ act in the space $L_p(\mathbb R^c,\mathbb C)$, $1\le p\le\infty$. We represent $\mathbb R^c$ as the union of the disjoint sets:
\begin{equation*}
\mathbb R^c=\bigcup_{m=(m_1,m_2,\dots,m_c)\in\mathbb Z^c}[m_1,m_1+1)\times[m_2,m_2+1)\times\dots\times[m_c,m_c+1)
\end{equation*}
and identify the space $L_p(\mathbb R^c,\mathbb C)$ with $l_p\bigl(\mathbb Z^c,L_p([0,1)^c,\mathbb C)\bigr)$. Let $T$ be the operator acting in $l_p\bigl(\mathbb Z^c,L_p([0,1)^c,\mathbb C)\bigr)$, which corresponds to $A$ in accordance with this identification.  We say that the operator $A$ belongs to the class $\mathoo S_{1,g}$ if the operator $T$ belongs to the class $\mathoo s_{1,g}$.

The inverse-closedness of some other classes of operators possessing good proper\-ties (like compactness) only locally was investigated in~\cite{Beltita-Beltita15e,Beltita-Beltita15,Farrell-Strohmer10,
Guseva-Kurbatov20-CAOT,Kurbatov99,Kurbatov-FDE01,Kurbatov-Kuznetsova16}.
The preservation of the rate of decrease of memory (i.~e. an estimate of the kind~\eqref{e:beta:0}) while passing to the inverse operator was studied by many authors~\cite{Baskakov90:eng,Baskakov97a:eng,Baskakov2004:eng,Blatov-Terteryan92:eng,
Demko77,Demko86,Demko-Moss-Smith84,Fendler-Grochenig-Leinert08,Gohberg-Kaashoek-Woerdeman89,
Grochenig10,Grochenig-Klotz10,Grochenig-Leinert06,Jaffard90,KurbatovFA90:eng,Kurbatov90:eng,Kurbatov99,Sjostrand95}.
To make our exposition as self-contained as possible, we reproduce some of these results with proofs and in a form convenient for our purposes.

In Section~\ref{s:B-algebras}, we recall terminology, notation, and some general facts
connected with Banach algebras. In Section~\ref{s:l{1,g}} we recall necessary properties of the weighted algebra $l_{1,g}(\mathbb Z^c,\mathoo B)$. Our proofs are essentially based on the Bochner--Phillips theorem; we present its formulation in Section~\ref{s:Bochner-Phillips}. We recall the definition and basic properties of nuclear operators in Section~\ref{s:nuclear operators}; in Section~\ref{s:nuclear operators:L_p}, we refine them to the case of operators acting in $L_p$. Section~\ref{s:locally nuclear:na} is devoted to the proof of Theorem~\ref{t:fin:convolution dominayed:nuclear}, which is our main result; in Section~\ref{s:locally nuclear:na:L_p}, it is specified to the case of operators acting in $L_p$ (Theorem~\ref{t:convolution dominayed:nuclear:2}).

\section{Banach algebras}\label{s:B-algebras}
In this paper, all linear spaces are considered over the field $\mathbb C$ of complex numbers.

An \emph{algebra}~\cite[ch.~1, \S~1]{Bourbaki_Theories_Spectrales:fr},~\cite[ch.~4, \S~1.13]{Hille-Phillips:eng},~\cite[ch.~10, \S~10.1]{Rudin-FA:eng} is a linear space $\mathoo B$ (over the field $\mathbb C$ of complex numbers) endowed with a multiplication possessing the properties
\begin{gather*}
	A(BC)=(AB)C,\\
	\alpha (AB)=(\alpha A)B=A(\alpha B),\\
	(A+B)C=AC+BC,\; A(B+C)=AB+AC.
\end{gather*}
If $\mathoo B$ is a normed space and
$$\Vert AB\Vert\le\Vert A\Vert\cdot\Vert B\Vert,$$
then $\mathoo B$ is called a \emph{normed algebra}. If a normed algebra is a complete (i.~e. Banach) space, then it is called a \emph{Banach algebra}.

Let $X$ be a Banach space. We denote by $\mathoo B(X)$ the Banach algebra of all bounded linear operators acting in $X$. It is the main example of a Banach algebra.
	
If an algebra $\mathoo B$ has an element $\mathbf1$ such that
\begin{gather*}
A\mathbf1=\mathbf1A=A
\end{gather*}
for all $A\in\mathoo B$,
the element $\mathbf1$ is called a \emph{unit}.
In this case, the algebra $\mathoo B$ is called a \emph{unital algebra} or an \emph{algebra with unit}. If, in addition, the algebra $\mathoo B$ is normed (Banach) and
\begin{equation*}
\Vert\mathbf1\Vert=1,
\end{equation*}
then $\mathoo B$ is called a \emph{normed {\rm(}Banach{\rm)} unital algebra}.

Let $\mathoo B$ be a unital algebra and $A\in\mathoo B$.
An element $B\in\mathoo B$ is called the \emph{inverse} of $A$ if
\begin{equation*}
AB=BA=\mathbf1.
\end{equation*}
The inverse of $A$ is denoted by the symbol $A^{-1}$. If an element $A$ has an inverse, it is called \emph{invertible} (in the algebra $\mathoo B$).

Let $\mathoo B$ be a unital algebra and $A\in\mathoo B$. The set of all $\lambda\in\mathbb C$ such that the element $\lambda \mathbf1-A$ is not invertible is called the \emph{spectrum} of $A$ (in the algebra $\mathoo B$) and denoted by $\sigma(A)$ or $\sigma_{\mathoo B}(A)$.
The complement $\rho (A)=\rho_{\mathoo B}(A)=\mathbb C\setminus\sigma(A)$ is called the \emph{resolvent set} of $A$. The function $R_\lambda=(\lambda 1-A)^{-1}$ is called the \emph{resolvent} of $A$; the domain of definition of the resolvent is the set $\rho (A)$.

\begin{proposition} [{\cite[Theorem 10.13]{Rudin-FA:eng}\label{p:inequality spectr and norm of operator}}]
The spectrum of any element $A$ of a unital Banach algebra is a closed non-empty subset of $\mathbb C$ which is contained in the closed circle of radius $\lVert A\rVert$ centered at zero.
\end{proposition}

A subset $\mathoo R$ of an algebra $\mathoo A$ is called a \emph{subalgebra} if $\mathoo R$ is stable under the algebraic operations (addition, scalar multiplication, and multiplication), i.~e. $A+B,\lambda A,AB\in\mathoo R$ for all $A,B\in\mathoo R$ and $\lambda\in\mathbb C$.
Obviously, a subalgebra is an algebra itself. It is also evident that the closure of a subalgebra (of a normed algebra) is again a subalgebra. If a subalgebra $\mathoo R$ of a unital algebra $\mathoo A$ contains the unit of the algebra $\mathoo A$, then $\mathoo R$ is called a \emph{unital subalgebra}.

A unital subalgebra $\mathoo R$ of a unital algebra $\mathoo B$ is called \emph{full}~\cite[ch.~1,~\S~1.4]{Bourbaki_Theories_Spectrales:fr} or \emph{inverse-closed}~\cite[p.~183]{Grochenig10} if every $B\in\mathoo R$ that is invertible in $\mathoo B$ is also invertible in $\mathoo R$. This definition is equivalent to the following one: for any $B\in\mathoo R$, the existence of $B^{-1}\in\mathoo B$ such that $BB^{-1}=B^{-1}B=\mathbf1$ implies that $B^{-1}\in\mathoo R$.

\begin{theorem}[{\rm\cite[ch.~10, \S~10.18]{Rudin-FA:eng},~{\rm\cite[ch.~3, \S~2., p. 29]{Bourbaki_Theories_Spectrales:fr}}}]\label{t:notfull spectrum:Rudin}
Let $\mathoo B$ be a unital Banach algebra and $\mathoo R$ be its closed unital subalgebra. For any $R\in\mathoo R$, the set $\sigma_{\mathoo R}(R)$ is the union of $\sigma_{\mathoo B} (R)$ and a {\rm(}possibly empty{\rm)} collection of bounded components of the set $\rho_{\mathoo B}(R)$. In particular, the boundary of $\sigma_{\mathoo R} (R)$ is contained in the boundary of $\sigma_{\mathoo B}(R)$.
\end{theorem}

Let $\mathoo B$ be a non-unital algebra. We consider the algebra $\widetilde {\mathoo B}$ consisting of all ordered pairs $(\alpha, A)$, where $\alpha\in\mathbb C$ and $A\in\mathoo B$, with the operations
\begin{align*}
	(\alpha, A)+(\beta, B)&=(\alpha+\beta,A+B),\\
	\lambda(\alpha,A)&=(\lambda \alpha, A),\\
	(\alpha, A)(\beta, B)&=(\alpha \beta,\alpha B+\beta A+AB).
\end{align*}
It is easy to see that $\widetilde {\mathoo B}$ is in fact an algebra and the element $(1,\mathbf0)$ is its unit. Clearly, $\mathoo B$ is isometrically isomorphic to the subalgebra of $\widetilde{\mathoo B}$ consisting of elements of the form $(0,A)$. The element $(\alpha, A)$ is usually denoted by $\lambda\mathbf1+A$.
If $\mathoo B$ is normed, a norm on $\widetilde{\mathoo B}$ can be defined by the formula
\begin{equation}\label{d:norm in alg with adjoint unit}
\lVert (\alpha, A)\rVert =|\alpha|+\lVert A\rVert.
\end{equation}
Clearly, $\widetilde{\mathoo B}$ is Banach provided that so is $\mathoo B$. The algebra $\widetilde{\mathoo B}$ is called~\cite{Bourbaki-Algebra1-3:eng}
the algebra derived from $\mathoo B$ by \emph{adjoining a unit element}.
If $\mathoo B$ is unital, we mean by $\widetilde{\mathoo B}$ the algebra $\mathoo B$ itself.

In a particular case where the algebra $\mathoo B$ is a subalgebra of a unital algebra $\mathoo A$, we identify the algebra $\widetilde{\mathoo B}$ with the subalgebra $\{\alpha \mathbf1_{\mathoo A}+B \in\mathoo A: B\in\mathoo B,\,\alpha\in\mathbb C\}$ of the algebra $\mathoo A$. We note that in this case the norm on $\widetilde {\mathoo B}$ induced by the imbedding into $\mathoo A$ is equivalent to the norm~\eqref{d:norm in alg with adjoint unit}.

Let $\mathoo A$ and $\mathoo B$ be algebras. The mapping $\varphi:\,\mathoo A\to\mathoo B$ is called~\cite[ch.~1, \S~1]{Bourbaki_Theories_Spectrales:fr} a \emph{morphism of algebras} if
\begin{align*}
	\varphi(A+B)&=\varphi(A)+\varphi(B),\\
	\varphi(\alpha A)&=\alpha\varphi(A),\\
	\varphi(AB)&=\varphi(A)\varphi(B)
\end{align*}
for all $A,B\in\mathoo A$ and $\alpha\in\mathbb C$. If the algebras $\mathoo A$ and $\mathoo B$ are unital and additionally
\begin{equation*}
\varphi(\mathbf1_{\mathoo A})=\mathbf1_{\mathoo B},
\end{equation*}
then $\varphi$ is called a \emph{morphism of unital algebras}. If the algebras $\mathoo A$ and $\mathoo B$ are normed (Banach) and the morphism $\varphi$ is continuous, then $\varphi$ is called a \emph{morphism of normed {\rm(}Banach{\rm)} algebras}. If $\varphi^{-1}$ exists and is a morphism of the same type, then $\varphi$ is called an \emph{isomorphism}; in this case, $\mathoo A$ and $\mathoo B$ are called \emph{isomorphic}.

\begin{proposition}[{\rm\cite[ch.~5, \S~2, Proposition 3]{Helemskii06:eng}}]\label{p:morphism of algebras}
Let $\mathoo A$ and $\mathoo B$ be unital algebras and $\varphi:\,\mathoo A\to\mathoo B$ be a morphism of unital algebras. If $A\in\mathoo A$ is invertible, then $\varphi{\rm(}A{\rm)}$ is also invertible.
\end{proposition}

A linear subspace $\mathoo J$ of a Banach algebra $\mathoo B$ is called a (\emph{two-sided{\rm)} ideal} if it possesses the property:
$AJ,JA\in\mathoo J$ for all $A\in\mathoo B$ and $J\in\mathoo J$. Clearly, each ideal is a subalgebra. For any ideal $\mathoo J$, the quotient space $\mathoo B/\mathoo J$ is an algebra. An ideal $\mathoo J$ is called \emph{proper} if $\mathoo J\neq\{0\}$ and $\mathoo J\neq\mathoo B$. It is easy to see that if ($\mathoo B$ is unital and) an ideal $\mathoo J$ contains an invertible element, then $\mathoo J=\mathoo B$.

\medskip
An algebra $\mathoo B$ is called \emph{commutative} if for any $A,B\in\mathoo B$,
\begin{equation*}
	AB=BA.
\end{equation*}

Let $\mathoo B$ be a commutative unital Banach algebra. A \emph{character} of the algebra $\mathoo B$ is~\cite[ch.~1, \S~1.5]{Bourbaki_Theories_Spectrales:fr} any morphism of algebras $\chi:\mathoo B \rightarrow \mathbb{C},$ i.~e. a map $\chi: \mathoo B \rightarrow \mathbb{C}$ satisfying the conditions
\begin{equation}\label{e:char prop}
\begin{split}
\chi(A+B)&=\chi(A)+\chi(B),\\
\chi(AB)&=\chi(A)\chi(B),\\
\chi(\lambda A)&=\lambda \chi(A),\\
\chi(\mathbf1_{\mathoo B})&=1_{\mathbb{C}}
\end{split}
\end{equation}
for all $A, B \in \mathoo B$ and $\lambda \in \mathbb C$. The set of all characters of the algebra $\mathoo B$ is denoted by the symbol $X(\mathoo B)$.

If $\mathoo B$ is a commutative non-unital Banach algebra, the last condition in~\eqref{e:char prop} is omitted.

If an algebra $\mathoo B$ is non-unital, we denote by the symbol $\chi_0$ the character $\chi_0:\,\mathoo B\to\mathbb C$ that equals zero on all elements $\mathoo B$. The \emph{zero character} $\chi_0$ exists only if the algebra $\mathoo B$ is non-unital.

\begin{proposition}[{\rm\cite[ch.~1, \S~3.1, Theorem 1]{Bourbaki_Theories_Spectrales:fr}}]\label{pr:norm character}
The norm of any character of a commutative unital Banach algebra equals 1.
\end{proposition}

\begin{proposition}\label{p:char and adj unit}
Any character of a non-unital algebra $\mathoo B$ can be continuously extended to a character of the algebra $\widetilde{\mathoo B}$ derived from $\mathoo B$ by adjoining a unit element{\rm;} the extension is given by the formula $\chi(\lambda\mathbf1+A)=\lambda+\chi(A)$. Conversely, the restriction of any character of the algebra $\widetilde{\mathoo B}$ to $\mathoo B$ is a character of the algebra $\mathoo B$. In particular, the zero character $\chi_0$ is the restriction of the character \mbox{$\lambda\mathbf1+A\mapsto\lambda$.}
\end{proposition}

We denote by $X(\mathoo B)$ the set of all \emph{non-zero} characters of a commutative (unital or non-unital) algebra $\mathoo B$. The set $X(\mathoo B)$ is called~\cite{Bourbaki_Theories_Spectrales:fr} the \emph{character space} of $\mathoo B$.

\begin{proposition}\label{pr:character-1}
Let $\mathoo B$ be a commutative unital algebra, and $\chi$ be its character. If an element $A\in\mathoo B$ is invertible, then
\begin{equation*}
\chi(A^{-1})=\frac 1{\chi(A)}.
\end{equation*}
\end{proposition}

\begin{proof}
The proof follows from properties~\eqref{e:char prop}.
\end{proof}

\begin{theorem}[{\rm\cite[p. 31, Theorem 3$'$]{Gelfand-Raikov-Shilov:eng}, \cite[p. 33, Proposition 3]{Bourbaki_Theories_Spectrales:fr}, \cite[Theorem 11.9(c)]{Rudin-FA:eng}}]\label{t:Gelfand}
Let $\mathoo B$ be a commutative unital Banach algebra. An element $b\in \mathoo B$ is invertible if and only if $\chi(b)\neq 0 $ for all characters $\chi$ of the algebra $\mathoo B$.
\end{theorem}

\section{Algebra $l_{1,g}(\mathbb Z^c)$}\label{s:l{1,g}}
In this section, we recall some results related to the weighted algebra $l_{1,g}(\mathbb Z^c)$. The closest detailed exposition with a discussion of motivation and history can be found in~\cite{Grochenig10}. Unfortunately, we need a slightly different formulation than in~\cite{Grochenig10}; therefore, we reproduce the main ideas with proofs in a form convenient for our aims.

We denote by $\mathbb Z^c$, $c\in\mathbb N$, the Cartesian product of $c$ copies of the group $\mathbb Z$ of integers.

\begin{definition}\label{def:weight}
A \emph{weight} on the group $\mathbb Z^c$ is an arbitrary function $g:\mathbb{Z}^c \to(0,+\infty)$. We always assume that the weight $g$ on $\mathbb Z^c$ under consideration possesses the properties:
\begin{enumerate}
	\item[{\rm(a)}] $g(0)=1$,
	\item[{\rm(b)}] $g(m+n)\leq  g(m)g(n)$,
	\item[{\rm(c)}] $g(-n)=g(n)$,
	\item[{\rm(d)}] $g(n)\geq1$,
 \item[{\rm(e0)}] for any $t\in\mathbb Z^c$, we have $\lim\limits_{\substack{n\in\mathbb Z\\n\to\infty}} \frac{\ln g(nt)}{n}=0$,
 \item[{\rm(e1)}] for any $t\in\mathbb Z^c$, we have $\lim\limits_{\substack{n\in\mathbb Z\\n\to\infty}} \sqrt[n]{g(nt)}=1$.
\end{enumerate}
\end{definition}

Clearly, (d) is a consequence of (a), (b), and (c).

\begin{proposition}\label{p:e0 <-> e1}
Assumptions {\rm(e0)} and {\rm(e1)} are equivalent.
\end{proposition}
\begin{proof}
Indeed,
\begin{equation*}
\lim\limits_{\substack{n\in\mathbb Z\\n\to\infty}} \sqrt[n]{g(nt)}=1
\Leftrightarrow
\lim\limits_{\substack{n\in\mathbb Z\\n\to\infty}} \ln\sqrt[n]{g(nt)}=\ln1
\Leftrightarrow
\lim\limits_{\substack{n\in\mathbb Z\\n\to\infty}}\frac {\ln g(nt)}{n}=0.\qed
\end{equation*}
\renewcommand\qed{}
\end{proof}

\begin{example}[{\rm\cite[Example~5.21]{Grochenig10},~\cite{Grochenig-Leinert06}}]\label{ex:weights}
Let $0\le b<1$, $a\ge0$ and $s,t\ge0$. Then the functions
\begin{align*}
g(n)&=1,\\
g(n)&=(1+|n|)^s,\\
g(n)&=e^{a|n|^b}(1+|n|)^s,\\
g(n)&=e^{a|n|^b}(1+|n|)^s\ln^t(e+|n|)
\end{align*}
satisfy assumptions (a)--(e) from Definition~\ref{def:weight}. It is clear that in this list, any previous example is a special case of the next one.
\end{example}

Let $g$ be a weight on $\mathbb Z^c$, and let $\mathoo B$ be a Banach algebra. \emph{The space} $l_{1,g}=l_{1,g}(\mathbb{Z}^c,\mathoo B)$ is the set of all families $a=\{a_m \in\mathoo B:\, m\in \mathbb{Z}^c\}$ such that
\begin{equation*}
\|a\|=\|a\|_{l_{1,g}}=\sum\limits_{m \in \mathbb{Z}^c} g(m)\|a_m\|<\infty.
\end{equation*}
We endow $l_{1,g}$ with the coordinate-wise operations of addition and multiplication by scalars; clearly, $l_{1,g}$ becomes a linear space.
If $g(m)\equiv 1$, the space $l_{1,g}=l_{1,g}(\mathbb{Z}^c,\mathoo B)$ coincides with the ordinary space $l_1=l_1(\mathbb Z^c, \mathoo B)$. Clearly, assumption (d) from Definition~\ref{def:weight} implies that $l_{1,g}\subseteq l_1$ for any admissible weight $g$.

\begin{proposition}[{\rm\cite[p.~196, Lemma 5.22]{Grochenig10}}]\label{p:l_{1,g} is an algebra}
Let assumptions (a) and (b) be fulfilled.
Then the space $l_{1,g}=l_{1,g}(\mathbb Z^c,\mathoo B)$ is a Banach algebra with respect to the operation of convolution
\begin{equation*}
(a\ast b)_k=\sum_{m\in \mathbb Z}a_{m}b_{k-m},
\end{equation*}
taken as multiplication. If\/ $\mathoo B$ is unital, so is $l_{1,g}${\rm ;} the unit of the algebra $l_{1,g}$ is the family $\delta=\{\delta_k\}$ defined by the formula
\begin{equation*}
\delta_k=\begin{cases}
\mathbf1 \quad \text{when } k=0,\\
0 \quad \text{when } k\neq 0.
\end{cases}
\end{equation*}
The algebra $l_{1,g}(\mathbb Z^c,\mathbb C)$ is commutative.
\end{proposition}
Let $k=1,\dots,c$.
We call a \emph{coordinate subgroup} of the group $\mathbb Z^c$ the subset $\mathbb Z^{c(k)}$ of all families $n\in\mathbb Z^c$ of the form
\begin{equation*}
n=(0,\dots,0,\underset{k}{n_k},0,\dots,0),
\end{equation*}
where $n_k\in\mathbb Z$ stands at position $k$.
Evidently, $\mathbb Z^{c(k)}$ forms a subgroup isomorphic to the group $\mathbb Z$.
We denote by $l_{1,g}(\mathbb Z^{c(k)},\mathoo B)$ the subspace of the space $l_{1,g}(\mathbb Z^c,\mathoo B)$ consisting of all families $a=\{a_n\}$ possessing the property $a_n=0$ for $n\notin\mathbb Z^{c(k)}$.

\medskip
We describe all characters of the algebra $l_{1,g}(\mathbb Z^c, \mathbb C)$.
Let $m=1,\dots,c$. We denote by $\varepsilon^{(m)} \in l_{1,g}(\mathbb Z^{c},\mathbb C)$ the family
\begin{gather*}\label{eq:posl eps:c}
\varepsilon^{(m)}_k=
\begin{cases}
1 \qquad \text{when }  k=(0,\dots,0,\underset{m}{1},0,\dots,0),\\
0 \qquad \text{when }  k\neq(0,\dots,0,\underset{m}{1},0,\dots,0),
\end{cases}
\end{gather*}
where $1$ stands at position $m$. We note that $\varepsilon^{(m)}\in l_{1,g}(\mathbb Z^{c(m)},\mathbb C)$.
By the definition, we put
\begin{equation*}
\varepsilon^{(0)}_k=
\begin{cases}
1 \qquad \text{when }  k=(0,\dots,0),\\
0 \qquad \text{when }  k\neq(0,\dots,0).
\end{cases}
\end{equation*}
Clearly, $\varepsilon^{(0)}$ is the unit element $\delta$ (Proposition~\ref{p:l_{1,g} is an algebra}) of the algebra $l_{1,g}(\mathbb Z^{c},\mathbb C)$. Obviously, the family
\begin{gather*}
\varepsilon^{-(m)}_k=
\begin{cases}
1 \qquad \text{when }  k=(0,\dots,0,\underset{m}{-1},0,\dots,0),\\
0 \qquad \text{when }  k\neq(0,\dots,0,\underset{m}{-1},0,\dots,0)
\end{cases}
\end{gather*}
is the inverse of the family $\varepsilon^{(m)}$.

For an arbitrary $n=(n_1,\dots,n_c)\in\mathbb Z^c$, we set
\begin{equation*}
\varepsilon^{n}=\prod_{m=1}^c(\varepsilon^{(m)})^{n_m}.
\end{equation*}
Clearly, the family $\varepsilon^{n}=\{\varepsilon^{n}_k:\,k\in\mathbb Z^c\}$ consists of the elements
\begin{equation}\label{e:epsilon n}
\varepsilon^{n}_k=
\begin{cases}
1 \qquad \text{when }  k=n,\\
0 \qquad \text{when }  k\neq n,
\end{cases}\qquad k\in\mathbb Z^c.
\end{equation}
In particular, $\varepsilon^0=\delta$.

\begin{proposition}\label{p:||eps||:c}
For each family $\varepsilon^{n}\in l_{1,g}(\mathbb Z^c,\mathbb C)$, $n=(n_1,\dots,n_c)\in\mathbb Z^c$, we have
\begin{equation*}
\lVert\varepsilon^{n}\rVert_{l_{1,g}}=g(n).
\end{equation*}
\end{proposition}

\begin{proof}
Indeed,
\begin{equation*}
\lVert\varepsilon^{n}\rVert_{l_{1,g}}=\sum\limits_{m \in \mathbb{Z}^c} g(m) |\varepsilon^{n}_m|=g(n).\qed
\end{equation*}
\renewcommand\qed{}
\end{proof}

\begin{corollary}\label{c:aSeries:c}
Each family $a=\{a_n\in\mathbb C:\,n\in \mathbb{Z}^c\}\in l_{1,g}(\mathbb Z^c, \mathbb C)$ admits the representation
\begin{equation}\label{eq:aSeries:c}
	a=\sum\limits_{n\in\mathbb Z^c} a_n\varepsilon^{n},
\end{equation}
with the series absolutely convergent in the norm of $l_{1,g}$.
\end{corollary}

\begin{proof}
Let $a\in l_{1,g}(\mathbb Z^c,\mathbb C)$. By definition, $a$ has the form
\begin{equation*}
a=\{a_n:\, n\in \mathbb {Z}^c\}
\end{equation*}
and
\begin{equation}\label{eq:aProperty:c}
\|a\|_{l_{1,g}}=\sum\limits_{m \in \mathbb Z^c}g(m)|a_m|.
\end{equation}
We verify that series~\eqref{eq:aSeries:c} converges absolutely. Indeed, we have
\begin{equation*}
\sum\limits_{n\in \mathbb Z^c}\| a_n \varepsilon^n\|= \sum\limits_{n \in \mathbb Z^c}|a_n|\cdot \| \varepsilon^n \|= \sum\limits_{n \in \mathbb Z^c}|a_n| g(n)=\|a\|_{l_{1,g}}<\infty.
\end{equation*}
We consider the remainder of series~\eqref{eq:aSeries:c}:
\begin{equation*}
a-\sum\limits_{n\in[-k,k]^c}a_n \varepsilon^n.
\end{equation*}

Applying definition~\eqref{eq:aProperty:c}, we have
\begin{equation*}
\biggl\lVert a-\sum\limits_{n \in[-k,k]^c} a_n \varepsilon^n\biggr\rVert=\sum\limits_{n \in \mathbb Z^c \setminus [-k,k]^c}g(n)|a_n|.
\end{equation*}
Since series~\eqref{eq:aProperty:c} converges for any order of summation, its remainder (for any order of summation) tends to zero. Hence, the sum of series~\eqref{eq:aSeries:c} is the family~$a$.
\end{proof}

For $u=(u_1,\dots,u_c)\in(\mathbb C\backslash\{0\})^c$ and $n=(n_1,\dots,n_c)\in\mathbb Z^c$, we set
\begin{equation*}
u^n=u_1^{n_1}\cdot\ldots\cdot u_c^{n_c}\in\mathbb C.
\end{equation*}
Clearly,
\begin{equation*}
|u^n|=|u_1|^{n_1}\cdot\ldots\cdot |u_c|^{n_c}.
\end{equation*}

\begin{proposition}\label{pr:character l{1,g}:c}
Any character $\chi$ of the algebra $l_{1,g}(\mathbb Z^c,\mathbb C)$ admits a representation
\begin{equation}\label{eq:character l{1,g}:c}
\chi(a)=\sum\limits_{(n_1,\dots,n_c)\in\mathbb Z^c} u_1^{n_1}\cdot\ldots\cdot u_c^{n_c}\,a_{n_1,\dots,n_c}=\sum_{n\in\mathbb Z^c}u^na_n,
\end{equation}
where $u=(u_1,\dots,u_c)\in(\mathbb C\backslash\{0\})^c$, $u_m=\chi(\varepsilon^{(m)})$, and an element $a\in l_{1,g}$ is given by~\eqref{eq:aSeries:c}. In particular, $\chi(\varepsilon^n)=u^n$.
\end{proposition}
\begin{proof}
Since series~\eqref{eq:aProperty:c} converges (Proposition~\ref{c:aSeries:c}) and $\chi$ is continuous (Proposition~\ref{pr:norm character}), we have
\begin{align*}
\chi(a)&=\chi\Bigl(\sum\limits_{n\in\mathbb Z^c} a_n\varepsilon^{n}\Bigr)
=\sum\limits_{n\in\mathbb Z^c} a_n\chi(\varepsilon^{n})\\
&=\sum\limits_{n\in\mathbb Z^c} a_n\chi\Bigl(\prod_{m=1}^c(\varepsilon^{(m)})^{n_m}\Bigr)
=\sum\limits_{n\in\mathbb Z^c} a_n\prod_{m=1}^c\bigl(\chi(\varepsilon^{(m)})\bigr)^{n_m}\\
&=\sum\limits_{n\in\mathbb Z^c} a_nu^n.\qed
\end{align*}
\renewcommand\qed{}
\end{proof}

\begin{proposition}\label{pr:norm character in algebra l{1,g}:c}
Let $\chi:l_{1,g}(\mathbb Z^c,\mathbb C)\to\mathbb C$ be a character. Then
\begin{equation*}
    \|\chi\|=\sup\limits_{n\in \mathbb Z^c}\frac{|u^n|}{g(n)},
\end{equation*}
where $u=\chi(\varepsilon)$.
\end{proposition}

\begin{proof}
We introduce the notation
\begin{equation*}
\gamma=\sup\limits_{n\in\mathbb Z^c}\frac{|u^n|}{g(n)}.
\end{equation*}
We show that $\gamma \le \|\chi \|$. To do this, we recall (Proposition~\ref{p:||eps||:c}) that $\|\varepsilon^n\|=g(n)$. At the same time, $\chi(\varepsilon^n)=u^n$. This implies $\|\chi\|\geq \sup\limits_{n\in\mathbb Z^c}\frac{|u^n|}{g(n)}$.
	
Conversely, we take an arbitrary $a\in l_{1,g}(\mathbb{Z}^c, \mathbb{C})$. By definition, we have $\|a\|=\sum\limits_{n \in \mathbb Z^c}|a_n|g(n)$. Therefore, by Proposition~\ref{pr:character l{1,g}:c},
\begin{align*}
\frac{|\chi(a)|}{\|a\|}=\frac{\Bigm|\sum\limits_{n \in \mathbb Z^c}a_n u^n\Bigm|}{\|a\|}\le\frac{\sum\limits_{n \in \mathbb Z^c}|a_n|\cdot |u^n|}{\|a\|}\le \frac{\sum\limits_{n \in \mathbb Z^c}|a_n|\gamma g(n)}{\sum\limits_{n \in \mathbb Z^c}|a_n|g(n)}=\gamma.
\end{align*}
Passing to the supremum over $a\neq 0$, we arrive at $\|\chi\|\le \gamma$.
\end{proof}

\begin{corollary}\label{cor:property for character in l{1,g}:c}
A family $u=(u_1,\dots,u_c)\in\mathbb (\mathbb C\backslash\{0\})^c$ generates a character of the algebra $l_{1,g}(\mathbb Z^c,\mathbb C)$ by formula~\eqref{eq:character l{1,g}:c} if and only if $u$ possesses the property
\begin{equation*}
|u^n|\le g(n),\qquad n\in \mathbb Z^c.
\end{equation*}
In this case, series~\eqref{eq:character l{1,g}:c} converges absolutely.
\end{corollary}

\begin{proof}
Let $\chi$ be a character of the algebra $l_{1,g}(\mathbb Z^c,\mathbb C)$.
We consider the elements $\varepsilon^{n}$ and $\varepsilon^{-n}$, see~\eqref{e:epsilon n}. From the identity
	\begin{equation*}
	\chi(\varepsilon^{n})\chi(\varepsilon^{-n})=\chi(\varepsilon^{n} \cdot \varepsilon^{-n})=\chi(\delta)=1,
	\end{equation*}
we obtain that $u^n=\chi(\varepsilon^{n})\neq 0$.
	
Let formula~\eqref{eq:character l{1,g}:c} define a character. By Propositions~\ref{pr:norm character} and~\ref{pr:norm character in algebra l{1,g}:c}, we have
	\begin{equation*}
	\frac{|u^n|}{g(n)}\le 1
	\end{equation*}
or $|u^n|\le g(n)$ for all $n\in\mathbb Z^c$.
	
	Conversely, let $|u^n|\le g(n)$ for all $n \in \mathbb Z^c$. We show that formula~\eqref{eq:character l{1,g}:c} defines a character.
First, we verify that series~\eqref{eq:character l{1,g}:c} converges absolutely. Indeed,
	\begin{equation*}
	\sum\limits_{n \in \mathbb Z^c}|a_n u^n|=\sum\limits_{n \in \mathbb Z^c}|a_n|\cdot |u^n|\le \sum\limits_{n \in \mathbb Z^c}|a_n| g(n)=\|a\|_{l_{1,g(\mathbb Z^c, \mathbb C)}}< \infty.
	\end{equation*}	
Next, we verify properties~\eqref{e:char prop}. Only the second property is not completely evident: 	
\begin{align*}
	\chi(a)\chi(b)&=\sum\limits_{m \in \mathbb Z^c}  a_m u^m \sum\limits_{n \in \mathbb Z^c}b_{n}u^{n}=\sum\limits_{m \in \mathbb Z^c}  a_m u^m \sum\limits_{k \in \mathbb Z^c}b_{k-m}u^{k-m}\\
	&=\sum\limits_{m \in \mathbb Z^c} \sum\limits_{k \in \mathbb Z^c} a_m u^m b_{k-m}u^{k-m}=\sum\limits_{k \in \mathbb Z^c} \biggl(\sum\limits_{m \in \mathbb Z^c} a_mb_{k-m}\biggr)u^{k-m}u^m\\
	&=\sum\limits_{k \in \mathbb Z^c} \biggl(\sum\limits_{m \in \mathbb Z^c}a_m b_{k-m} \biggr)u^k=\sum\limits_{k \in \mathbb Z^c} (a \ast b)_k u^k=\chi(a\ast b).\qed
	\end{align*}
\renewcommand\qed{}
\end{proof}

\begin{proposition}\label{p:ln g(n)/n}
	Let $g$ be a weight on $\mathbb Z$ and assumption (b) from Definition~\ref{def:weight} be fulfilled. Then the limits $\lim\limits_{n\to \pm\infty} \frac{\ln g(n)}{n}$ exist and
	\begin{align*}
		\lim\limits_{n\to + \infty} \frac{\ln g(n)}{n}&=\inf\limits_{n>0}\frac{\ln g(n)}{n},\\
		\lim\limits_{n\to - \infty} \frac{\ln g(n)}{n}&=\sup\limits_{n<0}\frac{\ln g(n)}{n}.
	\end{align*}
\end{proposition}

\begin{proof}
We begin with the first formula. We take an arbitrary $\varepsilon >0$ and choose a number $n_\varepsilon \in \mathbb N$ so that
	\begin{equation*}
		\frac {\ln g(n_\varepsilon)}{n_\varepsilon}\le \inf\limits_{n>0}\frac{\ln g(n)}{n}+\varepsilon.
	\end{equation*}
We take an arbitrary $k\in\mathbb N$ and represent it in the form $k=ln_\varepsilon+m$, where $m=0,1,\dots,n_\varepsilon-1$ and $l \in \mathbb Z$.
	By assumption (b), we have
	\begin{gather*}
		g(k)=g(ln_\varepsilon+m)\le \underbrace{g(n_\varepsilon)g(n_\varepsilon)\dots g(n_\varepsilon)}_{l \text{ times}} \cdot g(m)=[g(n_\varepsilon)]^l g(m).
	\end{gather*}
	Hence,
	\begin{equation*}
		\ln g(k)\le l \ln g(n_\varepsilon)+\ln g(m).
	\end{equation*}
From this inequality, it follows
	\begin{equation*}
		\frac{\ln g(k)}{k}=\frac{\ln g(k)}{ln_\varepsilon+m}\le \frac{l}{ln_\varepsilon+m}\ln g(n_\varepsilon)+\frac{\ln g(m)}{ln_\varepsilon+m}\underset{l\to+\infty}{\to} \frac{\ln g(n_\varepsilon)}{n_\varepsilon}.
	\end{equation*}
Therefore, for $k$ large enough, we have
	\begin{equation*}
		\frac{\ln g(k)}{k}\le\frac{\ln g(n_\varepsilon)}{n_\varepsilon}+\varepsilon< \underset{n>0}{\inf} \frac{\ln g(n)}{n}+2\varepsilon.
	\end{equation*}
	At the same time, by the definition of infimum, for all $k>0$ we have
	\begin{equation*}
		\underset{n>0}{\inf}\frac{\ln g(n)}{n}\le\frac{\ln g(k)}{k}.
	\end{equation*}
Consequently,
	\begin{equation*}
		\lim\limits_{k\to +\infty} \frac{\ln g(k)}{k}=\underset{n>0}{\inf} \frac{\ln g(n)}{n}.
	\end{equation*}
	
Now we prove the second formula. We take an arbitrary $\varepsilon >0$ and choose a number $n_\varepsilon \in - \mathbb N$ so that
	\begin{equation*}
		\frac {\ln g(n_\varepsilon)}{n_\varepsilon}\geq \sup\limits_{n<0}\frac{\ln g(n)}{n}-\varepsilon.
	\end{equation*}
We take an arbitrary $k \in - \mathbb N$ and represent it in the form $k=ln_\varepsilon+m$, where $m=0,-1,\dots,-n_\varepsilon+1$, $l\in\mathbb Z$. Arguing as above, we obtain
	\begin{equation*}
		\frac{\ln g(k)}{k}=\frac{\ln g(k)}{ln_\varepsilon+m}\geq \frac{l}{ln_\varepsilon+m}\ln g(n_\varepsilon)+\frac{\ln g(m)}{ln_\varepsilon+m}\underset{l \to +\infty}{\to} \frac{\ln g(n_\varepsilon)}{n_\varepsilon}.	
	\end{equation*}
Therefore, for negative $k$ large enough in absolute value, we have
	\begin{equation*}
		\frac{\ln g(k)}{k}\geq\frac{\ln g(n_\varepsilon)}{n_\varepsilon}-\varepsilon> \underset{n<0}{\sup} \frac{\ln g(n)}{n}-2\varepsilon.
	\end{equation*}
	At the same time,
	\begin{equation*}
		\underset{n<0}{\sup}\frac{\ln g(n)}{n}\geq\frac{\ln g(k)}{k}.
	\end{equation*}
	Hence,
	\begin{equation*}
		\lim\limits_{k\to -\infty} \frac{\ln g(k)}{k}=\underset{n<0}{\sup} \frac{\ln g(n)}{n}.\qed
	\end{equation*}
\renewcommand\qed{}
\end{proof}

We consider the multiplicative group
\begin{equation*}
\mathbb U=\{u\in\mathbb C:|u|=1\}.
\end{equation*}
We denote by $\mathbb U^c$ the corresponding Cartesian product. Evidently, $\mathbb U^c$ is a commuta\-ti\-ve group with respect to the componentwise multiplication.

\begin{theorem}[{\rm\cite[Theorem 5.24]{Grochenig10}}]\label{th: character in l1, l_{1,g}}
Let assumptions {\rm(a)--(e)} from Definition~\ref{def:weight} be fulfilled. Then formula~\eqref{eq:character l{1,g}:c} defines a character of the algebra $l_{1,g}(\mathbb Z^c,\mathbb C)$ if and only if $u\in\mathbb U^c$. In particular, the character space of the algebra $l_{1,g}(\mathbb Z^c,\mathbb C)$ does not depend on the weight $g$ {\rm(}provided $g$ satisfies assumptions {\rm(a)--(e)}{\rm)}.
\end{theorem}

\begin{proof}
Let $\chi$ be a character of the algebra $l_{1,g}(\mathbb{Z}^c, \mathbb{C})$. By Proposi\-tion~\ref{pr:character l{1,g}:c}, it has the form~\eqref{eq:character l{1,g}:c}. Moreover, by Corollary~\ref{cor:property for character in l{1,g}:c}, $|u^k|\le g(k)$ for all $k\in\mathbb Z^c$, or
	\begin{equation*}
	|u_1|^{k_1}\cdot|u_2|^{k_2}\cdot\ldots\cdot |u_c|^{k_c}\le g(k_1,k_2,\dots,k_c),\qquad k\in\mathbb Z^c.
	\end{equation*}
In particular,
	\begin{equation*}
	|u_1|^{k_1}\cdot|u_2|^{0}\cdot\ldots\cdot |u_c|^{0}\le g(k_1,0,\dots,0),\qquad k_1\in\mathbb Z,
	\end{equation*}
or
	\begin{equation*}
	|u_1|^{k_1}\le g(k_1,0,\dots,0),\qquad k_1\in\mathbb Z.
	\end{equation*}
The last estimate is equivalent to
\begin{equation*}
k_1 \ln|u_1| \le\ln g(k_1,0,\dots,0),
\end{equation*}
or
\begin{align}
\ln|u_1| &\le \frac{\ln g(k_1,0,\dots,0)}{k_1}&& \text{for $k_1>0$},\label{e:ineq1}\\
\ln|u_1| &\ge \frac{\ln g(k_1,0,\dots,0)}{k_1}&& \text{for $k_1<0$.}\label{e:ineq2}
\end{align}
Passing to the limit as $k_1\to+\infty$ in~\eqref{e:ineq1}, we obtain
	\begin{equation*}
		\ln |u_1| \le \lim\limits_{k_1\to + \infty} \frac{\ln g(k_1,0,\dots,0)}{k_1}=0.
	\end{equation*}
Similarly, passing to the limit as $k_1\to-\infty$ in~\eqref{e:ineq2}, we obtain
	\begin{equation*}
		\ln |u_1| \geq \lim\limits_{k_1 \to - \infty} \frac{\ln g(k_1,0,\dots,0)}{k_1}=0.
	\end{equation*}
Hence $\ln|u_1|=0$ or $|u_1|=1$.
	
In the same way, one proves that $|u_2|=\dots=|u_c|=1$.
	
Conversely, let $|u_1|=\ldots=|u_c|=1$. We show that formula~\eqref{eq:character l{1,g}:c} defines a character of the algebra $l_{1,g}(\mathbb Z^c, \mathbb C)$. We make use of Corollary~\ref{cor:property for character in l{1,g}:c} again.
The condition $|u|^k\le g(k)$ from Corollary~\ref{cor:property for character in l{1,g}:c} is fulfilled  by virtue of assumption (d) from Definition~\ref{def:weight}.
\end{proof}

\begin{corollary}[{\rm\cite[Corollary 5.27]{Grochenig10}}]\label{c:l_1,g is full:1:c}
Let assumptions {\rm(a)--(e)} from Definition~\ref{def:weight} be fulfilled. Then the subalgebra $l_{1,g}(\mathbb Z^c,\mathbb C)$ is full in the algebra $l_1(\mathbb Z^c,\mathbb C)$.	
\end{corollary}

\begin{proof}
Let an element $a\in l_{1,g}(\mathbb Z^c, \mathbb C)$ have an inverse $b \in l_1(\mathbb Z^c, \mathbb C)$. Then, by Proposition~\ref{p:morphism of algebras}, $\chi(b)\neq0$ for any character $\chi$ of the algebra $l_1(\mathbb Z^c, \mathbb C)$. By Theorem~\ref{th: character in l1, l_{1,g}}, the character spaces of $l_1(\mathbb Z^c, \mathbb C)$ and $l_{1,g}(\mathbb Z^c, \mathbb C)$ coincide. Therefore, $\chi(b)\neq0$ for any character $\chi$ of the algebra $l_{1,g}(\mathbb Z^c, \mathbb C)$. Consequently, by Theorem~\ref{t:Gelfand}, $b$ is invertible in~$l_{1,g}(\mathbb Z^c, \mathbb C)$.
\end{proof}

\section{The Bochner--Phillips theorem}\label{s:Bochner-Phillips}
In this section we recall the Bochner--Phillips theorem~\cite{Bochner-Phillips} in a form convenient for our exposition.

The \emph{{\rm(}algebraic{\rm)} tensor product} $X\otimes Y$ of linear spaces $X$ and $Y$ is~\cite[ch.~3, \S~1]{Bourbaki-Algebra1-3:eng},~\cite[ch. 2, \S~7]{Helemskii06:eng} the linear space of all formal sums
	\begin{equation*}
	z=\sum\limits_{k=1}^n x_k\otimes y_k, \qquad x_k \in X,\,y_k \in Y,
	\end{equation*}
in which the following expressions are identified:
	\begin{equation}\label{e:tensor equi}
	\begin{split}
	(x_1+x_2)\otimes y=x_1 \otimes y + x_2 \otimes y,\\
	x \otimes (y_1+y_2)=x\otimes y_1 + x \otimes y_2,\\
	\alpha (x \otimes y)= (\alpha x) \otimes y = x\otimes (\alpha y).
	\end{split}
	\end{equation}
	
Let $X$ and $Y$ be Banach spaces. A non-degenerate norm $\alpha(\cdot)=\|\cdot\|_\alpha$ on $X\otimes Y$ is called~\cite{Grothendieck66,Schatten50,Defant-Floret} a \emph{cross-norm} if 	\begin{equation*}
\|x \otimes y\|_\alpha= \|x\| \cdot \|y\|
\end{equation*}
for all $x\in X$ and $y\in Y$.

We denote by $X\otimes_\alpha Y$ the space $X\otimes Y$ endowed with the cross-norm $\alpha$. If $X$ and $Y$ are infinite-dimensional, then the space $X\otimes_\alpha Y$ is not complete. We denote by $X\overline{\otimes}_\alpha Y$ the completion of $X\otimes_\alpha Y$ and call it a \emph{topological tensor product} of $X$ and $Y$. The cross-norm on $X\otimes Y$ is not unique. So, different completions of $X\otimes Y$ are possible. In this paper, we deal only with the completion with respect to the \emph{projective cross-norm}
\begin{gather*}
\lVert v\rVert_\pi=\inf \Bigl\{\sum\limits_{k=1}^n \|e_k\| \cdot \|x_k\|:\,v=\sum\limits_{k=1}^n e_k \otimes x_k \Bigr\},
\end{gather*}
where the infimum is taken over all representations $v=\sum\limits_{k=1}^n e_k \otimes x_k$.

\begin{theorem}[{\rm\cite[ch. 1]{Grothendieck66},~\cite[ch. 4, \S~6]{Schatten50},~\cite[Theorem 1.7.4 (c)]{Kurbatov99}}]\label{t:GS}
	Let $X$ be a Banach space. Let $T$ be a locally compact topological space with a positive measure. Then
	\begin{gather*}
	X\overline{\otimes}_\pi L_1(T,\mathbb{C}) \simeq L_1(T,X).
	\end{gather*}
This isomorphism is natural {\rm(}in the sense that to the element $a\otimes f$ there corres\-ponds the function $t\mapsto af(t)${\rm)} and isometric.
\end{theorem}

\begin{corollary}
For any Banach space $X$, the canonical isometric isomorphism
\begin{equation*}
X\overline{\otimes}_\pi l_{1,g}(\mathbb Z^c,\mathbb{C})\simeq l_{1,g}(\mathbb Z^c,X)
\end{equation*}
holds.
\end{corollary}

According to Theorem~\ref{t:GS}, we denote by the symbol $a\otimes f$ not only the element of the tensor product, but also the corresponding function $t\mapsto af(t)$.

Let $\mathoo B$ be a Banach algebra.
We note that by virtue of the above agreement, the symbol $b=a_n\otimes\varepsilon^n$, where $\varepsilon^n$ is defined by~\eqref{e:epsilon n}, means both the vector in the tensor product and the family $b$ in $l_{1,g}(\mathbb Z^c,\mathoo B)$ consisting of the elements
\begin{equation*}
b_m=\begin{cases}
a_n,\qquad &\text{when } m=n,\\
0,\qquad &\text{when } m\neq n,
\end{cases}
\qquad m\in\mathbb{Z}^c.
\end{equation*}

\begin{proposition}[{\rm cf. Corollary~\ref{c:aSeries:c}}]\label{pr:aSeries:B}
Any element $a \in l_{1,g}(\mathbb Z^c,\mathoo B)\simeq\mathoo B\overline{\otimes}_\pi l_{1,g}(\mathbb{Z}^c,\mathbb{C})$ can be repre\-sen\-ted as the absolutely convergent series
	\begin{equation}\label{eq:aSeries:B^c}
		a=\sum\limits_{k\in\mathbb Z^c} a_k\otimes\varepsilon^k,
	\end{equation}
where $a_k\in\mathoo B$ are the elements of the family $a$ and $\varepsilon^k$ is defined by~\eqref{e:epsilon n}.
\end{proposition}

\begin{proof}
We take $a\in l_{1,g}(\mathbb Z^c,\mathoo B)$. By definition, $a$ is a family
\begin{equation*}
a=\{a_k\in\mathoo B: k\in\mathbb {Z}^c\},
\end{equation*}
with
\begin{equation}\label{eq:aProperty:B^c}
\|a\|_{l_{1,g}}=\sum\limits_{k\in \mathbb Z^c}g(k)\|a_k\|<\infty.
\end{equation}
We verify that series~\eqref{eq:aSeries:B^c} converges absolutely. Indeed, we have (Proposition~\ref{p:||eps||:c})
	\begin{equation*}
	 \sum\limits_{k\in \mathbb Z^c}\| a_k\otimes\varepsilon^k\|= \sum\limits_{k \in \mathbb Z^c}\|a_k\|\cdot\| \varepsilon^k \|= \sum\limits_{k \in \mathbb Z^c}\|a_k\| g(k)=\|a\|_{l_{1,g}}.
	\end{equation*}
We consider the remainder of series~\eqref{eq:aSeries:B^c}:
\begin{equation*}
	a-\sum\limits_{k\in[-n,n]^c}a_k\otimes\varepsilon^k.
\end{equation*}
By~\eqref{eq:aProperty:B^c}, we obtain
\begin{equation*}
 \biggl\lVert a-\sum\limits_{k \in[-n,n]^c} a_k\otimes \varepsilon^k\biggr\rVert_{l_{1,g}}=\sum\limits_{k \in \mathbb Z^c \setminus [-n,n]^c}g(k)\|a_k\|.
\end{equation*}
Since series~\eqref{eq:aProperty:B^c} converges for any order of summation, its remainder (for any order of summation) tends to zero. Hence, the sum of series~\eqref{eq:aSeries:B^c} is the family~$a$.
\end{proof}

Let $X$, $X_1$, $Y$, and $Y_1$ be linear spaces, and $A:\,X\to X_1$ and $B:\,Y\to Y_1$ be linear operators. Let $A\otimes B$ denote the linear operator that acts from $X\otimes Y$ to $X_1\otimes Y_1$ and is defined by the rule
\begin{equation*}
\bigl(A\otimes B\bigr)\Bigl(\sum\limits_{k=1}^n x_k\otimes y_k\Bigr)=
\sum\limits_{k=1}^n (Ax_k)\otimes (By_k).
\end{equation*}
It is easy to show that this definition is correct in the sense that it agrees with identifications~\eqref{e:tensor equi}.

We assume additionally that the spaces $X$, $X_1$, $Y$, and $Y_1$ are Banach, and the operators $A$ and $B$ are bounded. We endow $X\otimes Y$ and $X_1\otimes Y_1$ with the projective cross-norm. It is easy to show that
\begin{equation*}
\lVert A\otimes B\rVert_{X\otimes_\pi Y\to X_1\otimes_\pi Y_1}\le\lVert A\rVert\cdot\lVert B\rVert.
\end{equation*}
Thus, the operator $A\otimes B$ can be continuously extended to the operator
\begin{equation*}
 A\otimes B: X\overline{\otimes}_\pi Y\to X_1\overline{\otimes}_\pi Y_1,
\end{equation*}
which we denote by the same symbol $A\otimes B$.
For example, in Theorem~\ref{t:Bochner-Phillips:K} below, we have
\begin{equation*}
\mathbf1_{\mathoo B}\otimes\chi:\,
\mathoo B\overline{\otimes}_\pi\mathoo M\to\mathoo B\overline{\otimes}_\pi\mathbb C=
\mathoo B\otimes_\pi\mathbb C\simeq \mathoo B.
\end{equation*}

\begin{theorem}[{\rm\cite{Bochner-Phillips}, \cite[ch.~1, \S~1.7, Theorem 1.7.10]{Kurbatov99}}]\label{t:Bochner-Phillips:K}
Let $\mathoo B$ be a unital Banach algebra, $\mathoo M$ be a unital commutative Banach algebra, and $X=X(\mathoo M)$ be the character space of the algebra $\mathoo M$. An element $T \in\mathoo B\overline{\otimes}_\pi\mathoo M$ is left {\rm(}right{\rm)} invertible in the algebra $\mathoo B\overline{\otimes}_\pi\mathoo M$ if and only if the element
\begin{equation*}
\bigl(\mathbf1_{\mathoo B}\otimes\chi\bigr)(T)
\end{equation*}
is left {\rm(}right{\rm)} invertible  in the algebra $\mathoo B\otimes_\pi\mathbb C\simeq\mathoo B$ for all $\chi\in X(\mathoo M)$.
\end{theorem}

\begin{proposition}\label{p:1 otimes chi}
Let an element $a \in l_{1,g}(\mathbb Z^c, \mathoo B)$ be represented in the form~\eqref{eq:aSeries:B^c}, and a character $\chi$ of the algebra $l_{1,g}(\mathbb Z^c, \mathbb C)$ have the form~\eqref{eq:character l{1,g}:c}. Then
\begin{equation*}
(\mathbf1_{\mathoo B}\otimes\chi)(a)=\sum\limits_{k \in \mathbb Z^c} u^k a_k.	
\end{equation*}
\end{proposition}

\begin{proof}
By virtue of the definition of the tensor product of operators and by the continuity of the operators, we have
	\begin{equation*}
\bigl(\mathbf1_{\mathoo B}\otimes\chi\bigr)\Bigl(\sum\limits_{k\in\mathbb Z^c} a_k\otimes\varepsilon^k\Bigr)=\sum\limits_{k\in\mathbb Z} \mathbf1_{\mathoo B}a_k\otimes\chi(\varepsilon^k)		=\sum\limits_{k\in\mathbb Z^c} a_k\otimes u^k=\sum\limits_{k \in \mathbb Z^c} u^k a_k.\qed
	\end{equation*}
\renewcommand\qed{}
\end{proof}

\begin{theorem}\label{th:invertibility in l_{1,g}(Z,B)}
Let assumptions {\rm(a)--(e)} from Definition~\ref{def:weight} be fulfilled.
Let $\mathoo B$ be a unital Banach algebra.
An element $a \in l_{1,g}(\mathbb Z^c, \mathoo B)$ of the form~\eqref{eq:aSeries:B^c} is invertible if and only if the element
\begin{equation*}
	\sum\limits_{k \in \mathbb Z^c}u^k a_k
\end{equation*}
is invertible in the algebra $\mathoo B$ for all $u\in\mathbb{U}^c$.
\end{theorem}

\begin{proof}
The proof follows from Theorem~\ref{t:Bochner-Phillips:K} and Proposition~\ref{p:1 otimes chi}.
\end{proof}

\begin{corollary} \label{cor:l1gfull}
The subalgebra $l_{1,g}(\mathbb{Z}^c,\mathoo B)$ is full in the algebra $l_1(\mathbb{Z}^c,\mathoo B)$.
\end{corollary}

\begin{proof}
By Theorem~\ref{th:invertibility in l_{1,g}(Z,B)}, the invertibility of $a \in l_{1,g}(\mathbb{Z}^c, \mathoo B)$ in $l_{1,g}(\mathbb{Z}^c,\mathoo B)$ and the invertibility in $l_{1}(\mathbb{Z}^c,\mathoo B)$ mean the same.
\end{proof}

\section{Convolution dominated operators}\label{s:convolution dominayed}

Let $X$ be a Banach space. For $1\le p<\infty$, we denote by $l_p=l_p(\mathbb Z^c,X)$ the Banach space consisting of all families $\{x_n\in X:\, n\in\mathbb Z^c\}$ such that $\sum\limits_{n\in \mathbb Z^c}\lVert x_n\rVert^p<\infty$. We endow the space $l_p=l_p(\mathbb Z^c,X)$ with the usual norm
\begin{equation*}
\lVert x\rVert=\lVert x\rVert_{l_p}=\biggl(\sum\limits_{k \in \mathbb Z^c}\lVert x_k\rVert^p\biggr )^{\frac 1p}.
\end{equation*}
We denote by $l_\infty=l_\infty(\mathbb Z^c,X)$ the Banach space consisting of all bounded families $\{x_n\in X:\,n\in\mathbb Z^c\}$ with the norm
\begin{equation*}
\lVert x\rVert=\lVert x\rVert_{l_\infty}=\sup_{n\in\mathbb Z^c}\lVert x_n\rVert.
\end{equation*}

We call a \emph{matrix} (indexed by elements of the set $\mathbb Z^c$ with values in $\mathoo B(X)$) any family $\{\,a_{kl}\in\mathoo B(X):\,k,l\in\mathbb Z^c\,\}$. We say that an \emph{operator} $T\in\mathoo B\bigl(l_p(\mathbb Z^c,X)\bigr)$, $1\le p\le\infty$, \emph{is generated by a matrix} $\{\,a_{kl}\in\mathoo B(X):\,k,l\in\mathbb Z^c\,\}$ if
\begin{equation*}
(Tx)_k=\sum\limits_{m\in\mathbb Z^c} a_{kl}x_{l},\qquad k\in\mathbb Z^c,
\end{equation*}
for all $x\in l_p$,
where the series converges in norm. Obviously, not every matrix generates a (bounded) operator.
It is less obvious, that not every (bounded) operator is generated by some matrix,
see a counterexample in~\cite[Example 1.6.4]{Kurbatov99}. For our aims, it is convenient to change the enumeration of matrix elements; namely, we make the substitution $l=k-m$ and introduce the notation $b_{km}=a_{k,k-m}$:
\begin{equation}\label{e:operator D}
(Tx)_k=\sum\limits_{m\in\mathbb Z^c}b_{km}x_{k-m},\qquad k\in\mathbb Z^c.
\end{equation}

\begin{definition}\label{def:s_{1,g}}
Let $g$ be a weight on $\mathbb Z^c$ satisfying assumptions (a)--(e) from Definition~\ref{def:weight}.
We denote by $\mathoo s_{1,g}=\mathoo s_{1,g}\bigl(\mathbb Z^c,\mathoo B(X)\bigr)$ the set of all operators $T\in\mathoo B\bigl(l_p(\mathbb Z^c,X)\bigr)$, $1\le p\le\infty$, of the form~\eqref{e:operator D} such that the coefficients $b_{km}\in\mathoo B(X)$
satisfy the estimate
\begin{equation}\label{e:beta-est}
\lVert b_{km}\rVert_{\mathoo B(X)}\leq\beta_m
\end{equation}
for some $\beta\in l_{1,g}(\mathbb Z^c,\mathbb C)$. (Implicitly, we assume that $\beta_m\ge0$. Formally, it is explained as follows: two complex numbers can be compared only if they are real. Since $\lVert b_{km}\rVert\ge0$, so is $\beta_m$ by~\eqref{e:beta-est}.)
In other words, $\mathoo s_{1,g}=\mathoo s_{1,g}\bigl(\mathbb Z^c,\mathoo B(X)\bigr)$
consists of operators generated by matrices satisfying estimate~\eqref{e:beta-est} with $\beta\in l_{1,g}(\mathbb Z^c,\mathbb C)$. In the case $g(n)\equiv1$, $n\in\mathbb Z^c$, we use the simple symbol $\mathoo s_1$ instead of $\mathoo s_{1,g}$.
\end{definition}

\begin{proposition}\label{l:ClassSg}
Let $1\le p\le\infty$. Then series~\eqref{e:operator D} converges absolutely for all $x\in l_p(\mathbb Z^c,X)$ and defines a linear operator $T\in\mathoo B\bigl(l_p(\mathbb Z^c,X)\bigr)$. Furthermore,
	\begin{gather*}
\lVert T\rVert_{\mathoo B(l_p)} \leq \sum\limits_{m \in \mathbb Z^c}\beta_m=\lVert\beta\rVert_{l_1}\le\lVert\beta\rVert_{l_{1,g}}.
	\end{gather*}
\end{proposition}
\begin{proof}
For $p<\infty$, we have
	\begin{align*}
		\lVert Tx\rVert_{l_p} &=\sqrt[p]{\sum\limits_{k\in \mathbb Z^c}\Bigl\lVert \sum\limits_{m \in \mathbb Z^c} b_{km}x_{k-m} \Bigr\rVert_{X}^p}\\
&\le \sqrt[p]{\sum\limits_{k\in \mathbb Z^c}\Bigl(\sum\limits_{m \in \mathbb Z^c} \lVert b_{km} x_{k-m} \rVert_X\Bigr)^p}\\
		&\le \sqrt[p]{\sum\limits_{k\in \mathbb Z^c}\Bigl(\sum\limits_{m \in \mathbb Z^c} \lVert b_{km}\rVert_X \cdot \lVert x_{k-m} \rVert_{X}\Bigr)^p}\\
&\le \sqrt[p]{\sum\limits_{k\in \mathbb Z^c}\Bigl(\sum\limits_{m \in \mathbb Z^c} \beta_{m} \cdot \lVert x_{k-m} \rVert_{X}\Bigr)^p}\\
		&=\sqrt[p]{\sum\limits_{k\in \mathbb Z^c}\Bigl( (\beta \ast z)_k \Bigr)^p}\\
&=\lVert \beta*z \rVert_{l_p}\le \lVert \beta \rVert_{l_1} \cdot \lVert z\rVert_{l_p}
=\lVert \beta \rVert_{l_1} \cdot \lVert x\rVert_{l_p},
	\end{align*}
where $z$ is the numerical sequence $z_k=\lVert x_k\rVert$.

For $p=\infty$, we have
	\begin{align*}
		\lVert Tx \rVert_{l_\infty}&=\sup_{k\in\mathbb Z^c}\Bigl\lVert \sum\limits_{m \in \mathbb Z^c} b_{km}x_{k-m} \Bigr\rVert_{X}\\
&\le\sup_{k\in\mathbb Z^c}\sum\limits_{m \in \mathbb Z^c} \lVert b_{km} x_{k-m} \rVert_X \\
&\le \sup_{k\in\mathbb Z^c}\sum\limits_{m \in \mathbb Z^c} \lVert b_{km}\rVert_X \cdot \lVert x_{k-m} \rVert_{X}\\
&\le \sup_{k\in\mathbb Z^c}\sum\limits_{m \in \mathbb Z^c} \beta_{m} \cdot \lVert x_{k-m} \rVert_X\\
&=\sup_{k\in\mathbb Z^c}(\beta \ast z)_k \\
&=\lVert \beta \ast z \rVert_{l_\infty}\le \lVert \beta \rVert_{l_1} \cdot \lVert z\rVert_{l_\infty}=\lVert \beta \rVert_{l_1} \cdot \lVert x\rVert_{l_\infty},
	\end{align*}
where $z$ is the same.

From these estimates, it follows that $Tx\in l_p$ for all $x\in l_p$ and
$\lVert T\rVert_{l_p\to l_p}\le\lVert \beta \rVert_{l_1}$.

The linearity of the operator $T$ is evident.
\end{proof}

We describe one more representation of operators of the class $\mathoo s_{1,g}$.

For every $m\in\mathbb Z^c$, we consider the \emph{shift operator}
\begin{equation*}
(S_mx)_k=x_{k-m},\qquad k\in\mathbb Z^c.
\end{equation*}
Evidently, $S_m$ acts from $l_p(\mathbb Z^c,X)$ to itself, $1\le p\le\infty$, and has unit norm.

Let $a\in l_\infty\bigl(\mathbb Z^c,\mathoo B(X)\bigr)$, i.~e. $a=\{a_k\in\mathoo B(X):\,k\in\mathbb Z^c\}$ and
\begin{equation*}
\lVert a\rVert=\lVert a\rVert_{l_\infty}=\sup_{k\in\mathbb Z^c}\lVert a_k\rVert_{\mathoo B(X)}<\infty.
\end{equation*}
We associate with the family $a$ the \emph{multiplication operator}
\begin{equation*}
(Ax)_k=a_kx_k,\qquad k\in\mathbb Z^c,
\end{equation*}
that acts from $l_p(\mathbb Z^c,X)$ to itself. Evidently, for all $1\le p\le\infty$,
\begin{equation*}
\lVert A\rVert_{l_p\to l_p}=\lVert a\rVert_{l_\infty}.
\end{equation*}

\begin{proposition}\label{p:s_{1,g}:2}
Let an operator $T\in\mathoo s_{1,g}\bigl(\mathbb Z^c,\mathoo B(X)\bigr)$ have the form~\eqref{e:operator D}. Then the operator $T$ can be represented in the form
\begin{equation}\label{e:T=sum:2}
T=\sum_{m\in\mathbb Z^c}B_mS_m,
\end{equation}
where $B_m$ are the multiplication operators
\begin{equation*}
(B_mx)_k=b_{km}x_k,\qquad k\in\mathbb Z^c.
\end{equation*}
\end{proposition}
\begin{proof}
Evidently,
\begin{align*}
\lVert B_mS_m\rVert_{l_p\to l_p}&\le\lVert B_m\rVert_{l_p\to l_p}\cdot\lVert S_m\rVert_{l_p\to l_p}=\lVert B_m\rVert_{l_p\to l_p}
=\lVert b_m\rVert_{l_\infty}\\
&=\sup_{k\in\mathbb Z^c}\lVert b_{km}\rVert_{\mathoo B(X)}\le\beta_m,
\end{align*}
where $b_m=\{\,b_{km}:\,k\in\mathbb Z^c\,\}$. Therefore series~\eqref{e:T=sum:2} converges absolutely.

For any $x\in L_p$, we have
\begin{align*}
\biggl(\Bigl(\sum_{m\in\mathbb Z^c}B_mS_m\Bigr)x\biggr)_k&=\Bigl(\sum_{m\in\mathbb Z^c}B_mS_mx\Bigr)_k\\
&=\sum_{m\in\mathbb Z^c}(B_mS_mx)_k\\
&=\sum_{m\in\mathbb Z^c}b_{km}x_{k-m}\\
&=(Tx)_k.\qed
\end{align*}
\renewcommand\qed{}
\end{proof}

\begin{theorem}[{\rm\cite{KurbatovFA90:eng}, \cite[Theorem 2.2.7]{Kurbatov90:eng}, \cite[ch.~5, \S~5.2, Theorem 5.2.6]{Kurbatov99}}]\label{t:5.2.6}
The subalgebra $\mathoo s_{1}\bigl(\mathbb Z^c,\,\mathoo B(X)\bigr)$ is full in the algebra $\mathoo B\bigl(l_p(\mathbb Z^c,X)\bigr)$ for all $1\le p\le\infty$.
\end{theorem}

\begin{theorem}[{\rm\cite{Baskakov90:eng,Baskakov97a:eng,Baskakov2004:eng}}]\label{t:s_g is full in s}
Let assumptions {\rm(a)--(e)} from Definition~\ref{def:weight} be fulfilled.
Then the subalgebra $\mathoo s_{1,g}\bigl(\mathbb Z^c,\mathoo B(X)\bigr)$ is full in the algebra $\mathoo B\bigl(l_p(\mathbb Z^c,X)\bigr)$ for all $1\le p\le\infty$.
\end{theorem}
For the completeness of the exposition, we give a proof of Theorem~\ref{t:s_g is full in s}.
\begin{proof}
By Theorem~\ref{t:5.2.6}, it suffices to verify that the subalgebra $\mathoo s_{1,g}\bigl(\mathbb Z^c,\mathoo B(X)\bigr)$ is full in the algebra $\mathoo s_{1}\bigl(\mathbb Z^c,\mathoo B(X)\bigr)$ for all $1\le p\le\infty$.

We assume that an operator $T\in\mathoo s_{1,g}\bigl(\mathbb Z^c,\mathoo B(X)\bigr)$ is invertible in the algebra $\mathoo s_{1}\bigl(\mathbb Z^c,\mathoo B(X)\bigr)$. We show that actually $T^{-1}\in\mathoo s_{1,g}\bigl(\mathbb Z^c,\mathoo B(X)\bigr)$.

We denote briefly the operator $T^{-1}$ by $D$. By virtue of Theorem~\ref{t:5.2.6}, we have $D\in\mathoo s_{1}\bigl(\mathbb Z^c,\mathoo B(X)\bigr)$. So, by Proposition~\ref{p:s_{1,g}:2}, we can represent the operator $D$ as
\begin{equation*}
D=\sum_{m\in\mathbb Z^c}A_mS_m,
\end{equation*}
where $(A_mx)_k=a_{km}x_k$ are multiplication operators,
and $\sum_{m\in\mathbb Z^c}\lVert A_m\rVert<\infty$. The equalities $TD=\mathbf1$ and $DT=\mathbf1$ are equivalent to
\begin{equation}\label{e:BD=1}
\begin{split}
\sum_{m\in\mathbb Z^c}B_mS_mA_{k-m}S_{k-m}&=
\begin{cases}
\mathbf1_{\mathoo B(l_p(\mathbb Z^c,X))} & \text{when }k=0, \\
0 & \text{when }k\neq0,
\end{cases}\qquad k\in\mathbb Z^c,\\
\sum_{m\in\mathbb Z^c}A_mS_mB_{k-m}S_{k-m}&=
\begin{cases}
\mathbf1_{\mathoo B(l_p(\mathbb Z^c,X))} & \text{when }k=0, \\
0 & \text{when }k\neq0,
\end{cases}
\qquad k\in\mathbb Z^c.
\end{split}
\end{equation}
We consider the families $\mathfrak t=\{B_mS_m\in\mathoo B\bigl(l_p(\mathbb Z^c,X)\bigr):\, m\in \mathbb Z^c\}$ and $\mathfrak d=\{D_mS_m\in\mathoo B\bigl(l_p(\mathbb Z^c,X)\bigr):\, m\in \mathbb Z^c\}$. We interpret them as elements of the Banach algebra $l_{1}\bigl(\mathbb Z^c,\,\mathoo B(l_p(\mathbb Z^c,X))\bigr)$. Moreover, by the assumption, we have $\mathfrak t\in l_{1,g}\bigl(\mathbb Z^c,\mathoo B(l_p(\mathbb Z^c,X))\bigr)$.

By the definition of multiplication in the algebra $l_{1}\bigl(\mathbb Z^c,\,\mathoo B(l_p(\mathbb Z^c,X))\bigr)$ and by~\eqref{e:BD=1}, we have
\begin{align*}
\mathfrak t\mathfrak d&=\Bigl\{\sum_{m\in\mathbb Z^c}B_mS_mA_{k-m}S_{k-m}\in\mathoo B\bigl(l_p(\mathbb Z^c,X)\bigr):\, k\in \mathbb Z^c\Bigr\}=\mathbf1_{l_{1}(\mathbb Z^c,\mathoo B(l_p(\mathbb Z^c,X)))},\\
\mathfrak d\mathfrak t&=\Bigl\{\sum_{m\in\mathbb Z^c}A_mS_mB_{k-m}S_{k-m}\in\mathoo B\bigl(l_p(\mathbb Z^c,X)\bigr):\, k\in \mathbb Z^c\Bigr\}=\mathbf1_{l_{1}(\mathbb Z^c,\mathoo B(l_p(\mathbb Z^c,X)))}.
\end{align*}
Thus, the family $\mathfrak d$ is the inverse of $\mathfrak t$ in the algebra $l_{1}\bigl(\mathbb Z^c,\mathoo B(l_p(\mathbb Z^c,X))\bigr)$. Then, by Corollary~\ref{cor:l1gfull}, $\mathfrak d\in l_{1,g}\bigl(\mathbb Z^c,\mathoo B(l_p(\mathbb Z^c,X))\bigr)$. This means that $T^{-1}\in\mathoo s_{1,g}\bigl(\mathbb Z^c,\mathoo B(X)\bigr)$.
\end{proof}

\section{Nuclear operators}\label{s:nuclear operators}
Let $X$ be a Banach space and $X^*$ be its conjugate. An operator $A\in\mathoo B(X)$ is called~\cite{Grothendieck66,Pietch78:eng,Ruston51a,Ruston51b} \emph{nuclear} if it can be represented in the form
\begin{equation}\label{e:repr of A}
Ax=\sum_{i=1}^\infty a_i(x)y_i,
\end{equation}
where $y_i\in X$, $a_i\in X^*$, and
\begin{equation*}
\sum_{i=1}^\infty\lVert a_i\rVert\cdot\lVert y_i\rVert<\infty.
\end{equation*}
It is usually written briefly as
\begin{equation*}
A=\sum\limits_{i=1}^\infty a_i\otimes y_i.
\end{equation*}
		
We denote the set of all nuclear operators $A\in\mathoo B(X)$ by the symbol $\mathfrak S_1(X)$. We set
\begin{equation}\label{e:norm in S_1}
\lVert A\rVert_{\mathfrak S_1}=\inf\sum_{i=1}^\infty\lVert a_i\rVert\cdot\lVert y_i\rVert,
\end{equation}
where the infimum is taken over all representations of the operator $A$ in the form~\eqref{e:repr of A}.
It is interesting to note that the natural mapping from $X^* \overline{\otimes}_\pi X$ to $\mathfrak S_1(X)$ is not injective in general~\cite[p.~34]{Defant-Floret}. Clearly,
\begin{equation}\label{e:norm S_1>norm}
\lVert A\rVert_{\mathoo B(X)}\le\lVert A\rVert_{\mathfrak S_1},\qquad A\in\mathfrak S_1(X).
\end{equation}

\begin{proposition}[{\rm\cite[6.3.2]{Pietch78:eng}}]\label{p:Pietch 6.3.2}
The set $\mathfrak S_1(X)$ is an ideal in $\mathoo B(X)$. Moreover,
\begin{equation*}
\lVert JA\rVert_{\mathfrak S_1(X)},\;\lVert AJ\rVert_{\mathfrak S_1(X)}\le
\lVert J\rVert_{\mathfrak S_1(X)}\lVert A\rVert_{\mathoo B(X)},\qquad
J\in\mathfrak S_1(X),\,A\in\mathoo B(X).
\end{equation*}
The ideal $\mathfrak S_1(X)$ is complete with respect to the norm~\eqref{e:norm in S_1}.
\end{proposition}

\begin{proposition}[{\rm\cite[6.3.1 and 1.11.2]{Pietch78:eng}}]\label{p:Pietch 1.11.2}
Any nuclear operator is compact.
\end{proposition}

\begin{corollary}\label{c:nuclear is proper}
If the space $X$ is infinite-dimensional, then the ideal $\mathfrak S_1(X)$ is proper.
\end{corollary}
\begin{proof}
The fact~\cite[ch. IV \S~3, Theorem 5]{Dunford-Schwartz-I:eng} that the closed unit ball is not compact in any infinite-dimensional Banach space implies that a compact operator can not be invertible.
Now the proof follows from Proposition~\ref{p:Pietch 1.11.2}.
\end{proof}

We denote by $\widetilde{\mathfrak S_1(X)}$ the ideal $\mathfrak S_1(X)$ with an adjoint unit. We realize $\widetilde{\mathfrak S_1(X)}$ as a subalgebra of $\mathoo B(X)$.

The following theorem is probably known. But we have not found an appropriate reference.
\begin{theorem}\label{t:nuclear is full}
Let $X$ be a Banach space. Then $\widetilde{\mathfrak S_1(X)}$ is a full subalgebra of the algebra $\mathoo B(X)$.
\end{theorem}

\begin{proof}
If $X$ is finite-dimensional, $\widetilde{\mathfrak S_1(X)}=\mathoo B(X)$. Therefore, without loss of generality, we may assume that $X$ is infinite-dimensional.

Let $J\in\mathfrak S_1(X)$, $\lambda\in\mathbb C$, and the operator $\lambda\mathbf1-J$ be invertible in the algebra $\mathoo B(X)$. We show that $(\lambda\mathbf1-J)^{-1}\in\widetilde{\mathfrak S_1(X)}$.

Since $\mathfrak S_1(X)$ is proper (Corollary~\ref{c:nuclear is proper}), $\lambda\neq0$. We set $\nu=\frac1{\lambda}$. Clearly, it is enough to prove that the invertibility of $\mathbf1-\nu J$ in the algebra $\mathoo B(X)$ implies that $(\mathbf1-\nu J)^{-1}\in\widetilde{\mathfrak S_1(X)}$.

By Proposition~\ref{p:inequality spectr and norm of operator}, since the algebra $\widetilde{\mathfrak S_1(X)}$ is complete (in its norm~\eqref{e:norm in S_1}), there exists $\mu\neq0$ (sufficiently small) such that the operator $\mathbf1-\mu J$ is invertible in $\widetilde{\mathfrak S_1(X)}$.

Since any nuclear operator is compact (Proposition~\ref{p:Pietch 1.11.2}), the spectrum of $J$ in $\mathoo B(X)$ is~\cite[ch. VII, \S~4, Theorem 5]{Dunford-Schwartz-I:eng} denumerable and has no point of accumulation in $\mathbb C$ except possibly zero. Hence, the resolvent set of $J$ is arcwise connected. Therefore, there exists a continuous function $z:\,[0,1]\to\mathbb C\setminus\{0\}$ such that $z(0)=\mu$ and $z(1)=\nu$, and the operator $\mathbf1-z(t)J$ is invertible in $\mathoo B(X)$ for all $t\in[0,1]$.

We set
\begin{equation*}
M=\max_{t\in[0,1]}\bigl\lVert(\mathbf1-z(t)J)^{-1}\bigl\rVert_{\mathoo B(X)}.
\end{equation*}
We take points $0=t_0<t_1<\dots<t_n=1$ such that
\begin{equation*}
\max_{k=1,\dots,n}\bigl|z(t_{k})-z(t_{k-1})\bigr|\cdot\lVert J\rVert_{\mathfrak S_1(X)}<\frac1M.
\end{equation*}
We show that $\bigl(\mathbf1-z(t_k)J\bigr)^{-1}\in\widetilde{\mathfrak S_1(X)}$ for all $k=0,1,\dots,n$; in particular, $(\mathbf1-\nu J)^{-1}=\bigl(\mathbf1-z(t_n)J\bigr)^{-1}\in\widetilde{\mathfrak S_1(X)}$.
We proceed by induction on $k$. For $k=0$, the statement is true by the assumption. Let the statement be true for $k-1$. We show that it is true for $k$. We make use of the representation
\begin{equation*}
\mathbf1-z(t_k)J=\bigl(\mathbf1-z(t_{k-1})J\bigr)-\bigl(z(t_{k-1}-z(t_{k})\bigr)J.
\end{equation*}
We consider the series
\begin{equation}\label{e:series}
A^{-1}+A^{-1}BA^{-1}+A^{-1}BA^{-1}BA^{-1}+\ldots,
\end{equation}
where $A=\mathbf1-z(t_{k-1})J$ and $B=\bigl(z(t_{k-1})-z(t_{k})\bigr)J$. By Proposition~\ref{p:Pietch 6.3.2}, we have
\begin{equation*}
\lVert A^{-1}B\rVert_{\mathfrak S_1(X)}\le\lVert A^{-1}\rVert_{\mathoo B(X)}\cdot\lVert B\rVert_{\mathfrak S_1(X)}.
\end{equation*}
Since
\begin{align*}
\lVert A^{-1}\rVert_{\mathoo B(X)}&=\bigl\lVert\bigl(\mathbf1-z(t_{k-1})J\bigr)^{-1}\bigl\rVert_{\mathoo B}\le M,\\
\lVert B\rVert_{\mathfrak S_1(X)}&=\bigl\lVert(z(t_{k-1}-z(t_{k}\bigr)J\bigl\rVert_{\mathfrak S_1(X)}<\frac1M,
\end{align*}
this estimate gives
\begin{equation*}
\lVert A^{-1}B\rVert_{\mathfrak S_1(X)}<M\frac1M=1,
\end{equation*}
which implies that series~\eqref{e:series} converges in the norm of $\widetilde{\mathfrak S_1(X)}$ and thus defines an element of $\widetilde{\mathfrak S_1(X)}$. It is (well-known and) straightforward to verify that the sum of series~\eqref{e:series} is $(A-B)^{-1}=\bigl(\mathbf1-z(t_k)J\bigr)^{-1}$. Thus, $\bigl(\mathbf1-z(t_{k-1})J\bigr)^{-1}\in\widetilde{\mathfrak S_1(X)}$
implies $\bigl(\mathbf1-z(t_k)J\bigr)^{-1}\in\widetilde{\mathfrak S_1(X)}$
for all $k=1,\dots,n$. In particular, it follows that $(\mathbf1-\nu J)^{-1}\in\widetilde{\mathfrak S_1(X)}$.
\end{proof}

\section{Nuclear operators in $L_p$}\label{s:nuclear operators:L_p}
We denote by $\lambda$ the Lebesgue measure on $\mathbb R^c$.
Let $E\subseteq\mathbb R^c$ be a measurable subset.
We denote the integral of a summable function $x:\,E\to\mathbb C$ with respect to the Lebesgue measure $\lambda$ by the symbol $\int_E x(t)\,d\lambda(t)$ or simply $\int_E x(t)\,dt$.

We denote by $\mathscr L_p=\mathscr L_p(E,\mathbb C)$, $1\le p<\infty$, the space
of all measurable functions $u:\,E\to\mathbb C$ bounded in the semi-norm
\begin{equation*}
\lVert u\rVert=\lVert u\rVert_{L_p}=\Bigl(\int_{E}|u(t)|^p\,dt\Bigr)^{1/p},
\end{equation*}
and we denote by $\mathscr L_\infty=\mathscr L_\infty(E,\mathbb C)$ the space
of all measurable essentially bounded functions $u:\,E\to\mathbb C$ with the
semi-norm
\begin{equation*}
\Vert u\Vert=\Vert u\Vert_{L_\infty}=\esssup_{t\in E}|u(t)|.
\end{equation*}
Sometimes it is convenient to admit that functions $u\in\mathscr L_p$ may be undefined on a negligible (i.~e. of measure zero) set.
Finally, we denote by $L_p=L_p(E)=L_p(E,\mathbb C)$, $1\le p\le\infty$, the Banach
space of all classes of functions $u\in\mathscr L_p$ with the identification almost
everywhere. Usually they do not distinguish the spaces $\mathscr L_p$ and $L_p$. For more details, see~\cite{Bourbaki-Int1-6:eng}.

Numbers $p,q\in [1,+\infty]$ connected by the relation $\frac1p+\frac1q=1$ are called~\cite[ch. IV, \S~6.4]{Bourbaki-Int1-6:eng} \emph{conjugate exponents}.

\begin{proposition}[{\rm\cite[ch.~4, \S~6, Corollary 4 and Proposition 3]{Bourbaki-Int1-6:eng}}]\label{p:Holder}
Let $p,q\in [1,+\infty]$ be conjugate exponents. Then for any functions $x\in\mathcal L_p(E,\mathbb C)$ and $y\in\mathcal L_q(E,\mathbb C)$, we have $xy\in\mathcal L_1(E,\mathbb C)$ and
\begin{equation*}
\Bigl|\int_Ex(t)y(t)\,dt\Bigr|\le\lVert x\rVert_{L_p}\cdot\lVert y\rVert_{L_q}.
\end{equation*}
Moreover,
\begin{align*}
\lVert x\rVert_{L_p}&=\sup\Bigl\{\,\Bigl|\int_Ex(t)y(t)\,dt\Bigr|:\,\lVert y\rVert_{L_q}\le1\,\Bigr\},\\
\lVert y\rVert_{L_q}&=\sup\Bigl\{\,\Bigl|\int_Ex(t)y(t)\,dt\Bigr|:\,\lVert x\rVert_{L_p}\le1\,\Bigr\}.
\end{align*}
\end{proposition}

\begin{proposition}\label{p:a.e. conv}
Let $x_i\in\mathcal L_p(E,\mathbb C)$, $1\le p\le\infty$, and $\sum_{i=1}^\infty\lVert x_i\rVert_{L_p}<\infty$. Then for almost all $t\in E$, the series $\sum_{i=1}^\infty x_i(t)$ converges absolutely{\rm;} we denote the sum of the series by $x(t)${\rm;} thus, we obtain a function $x:\,E\to\mathbb C$ defined almost everywhere. It is claimed that $x\in\mathcal L_p(E,\mathbb C)$ and the series $\sum_{i=1}^\infty x_i(t)$ converges to $x$ in the $L_p$-norm.
\end{proposition}
\begin{proof}
The case $1\le p<\infty$ is carried through in~\cite[ch.~4, \S~3, Proposition 6]{Bourbaki-Int1-6:eng}, the case $p=\infty$ is actually analyzed in~\cite[ch.~4, \S~6]{Bourbaki-Int1-6:eng}. We describe it in more detail. So, let $p=\infty$. We denote by $F_i$ the set of all points $t$ such that $|x_i(t)|<2\lVert x_i\rVert_{L_\infty}$. By the definition of the space $\mathcal L_\infty$, the sets $F_i$ are of full measure (i.~e. having a negligible complement). Therefore $\cap_{i=1}^\infty F_i$ is a set of full measure as well. Obviously, the series $\sum_{i=1}^\infty x_i(t)$ converges for all $t\in\cap_{i=1}^\infty F_i$.
\end{proof}

\begin{proposition}[{\rm\cite[ch.~5, \S~5, Theorem 4]{Bourbaki-Int1-6:eng}}]\label{p:conj of L_p}
Let $p,q\in [1,+\infty]$ be conjugate exponents, with $p<\infty$. Then the conjugate space of $L_p(E,\mathbb C)$ is naturally isomet\-ri\-cally isomorphic to $L_q(E,\mathbb C)${\rm;} namely, each bounded linear functional on $L_p(E,\mathbb C)$ has the form
\begin{equation*}
f(x)=\int_Ex(t)y(t)\,dt,
\end{equation*}
where $y\in\mathcal L_q(E,\mathbb C)${\rm;} besides, $\lVert f\rVert=\lVert y\rVert_{L_q}$.
\end{proposition}

\begin{proposition}[{\rm\cite[ch. 5, \S~8.4, Theorem 1]{Bourbaki-Int1-6:eng}}]\label{p:Fubini}
Let $z\in\mathcal L_1(E \times E,\mathbb C)$. Then the function $s\mapsto z(t,s)$ is integrable for almost all $t\in E$, the function $t\mapsto\int_{E}z(t,s)\,ds$ is also integrable, and
\begin{equation*}
\iint_{E\times E}z(t,s)\,dt\,ds
=\int_{E}dt\int_{E}z(t,s)\,ds.
\end{equation*}
\end{proposition}

\begin{corollary}\label{c:Fubini}
Let $N\subset E\times E$ be negligible. Then the set $N_t=\{\,s\in E:\,(t,s)\in N\,\}$ is negligible for almost all $t\in E$.
\end{corollary}
\begin{proof}
It is enough to apply Proposition~\ref{p:Fubini} to the characteristic function of the set $N$.
\end{proof}

\begin{proposition}\label{p:L_p<L_1}
Let $E\subseteq\mathbb R^c$ be a summable set, with $\lambda(E)=M<\infty$. Let $p,q\in [1,+\infty]$ be conjugate exponents. Then $\mathcal L_p(E,\mathbb C)\subseteq\mathcal L_1(E,\mathbb C)$, and for all $x\in\mathcal L_p(E,\mathbb C)$ we have
\begin{equation*}
\lVert x\rVert_{L_1}\le M^{\frac1{q}}\lVert x\rVert_{L_p}.
\end{equation*}
\end{proposition}
\begin{proof}
We denote by 1 the function that is identically equal to one. Obviously, $1\in\mathcal L_q(E,\mathbb C)$, and $\lVert 1\rVert_{L_q}=M^{\frac1{q}}$ when $q<\infty$, and $\lVert 1\rVert_{L_\infty}=1$. By Proposition~\ref{p:Holder}, for any function $x\in\mathcal L_p(E,\mathbb C)$, we have $x=x\cdot1\in\mathcal L_1$ and
\begin{equation*}
\lVert x\rVert_{L_1}=\lVert x\cdot1\rVert_{L_1}\le\lVert x\rVert_{L_p}\cdot\lVert 1\rVert_{L_q}
=M^{\frac1{q}}\lVert x\rVert_{L_p}.\qed
\end{equation*}
\renewcommand\qed{}
\end{proof}

The following theorem should be known, but we have not also found a relevant reference.

\begin{theorem}\label{t:nuclear oper in L_p}
Let $p,q\in [1,+\infty]$ be conjugate exponents, with $p<\infty$. Let $E\subseteq\mathbb R^c$ be a summable set, with $\lambda(E)=M<\infty$. Let an operator $A\in\mathfrak S_1\bigl(L_p(E,\mathbb C)\bigr)$ be represented in the form
\begin{equation*}
A=\sum\limits_{i=1}^\infty a_i\otimes y_i,
\end{equation*}
where $y_i\in\mathcal L_p(E,\mathbb C)$ and $a_i\in\mathcal L_q(E,\mathbb C)$, with
\begin{equation}\label{e:est:nuclear}
\sum_{i=1}^\infty\lVert a_i\rVert_{L_q}\;\lVert y_i\rVert_{L_p}<\infty.
\end{equation}
Then the series
\begin{equation*}
k(t,s)=\sum_{i=1}^\infty a_i(s)y_i(t)
\end{equation*}
absolutely converges in the norm of $L_1$ {\rm(}and consequently, by Proposition~\ref{p:a.e. conv}, converges almost everywhere on $E\times E${\rm)} to a function $k\in\mathcal L_1(E\times E,\mathbb C)$. Moreover,
\begin{equation*}
\lVert k\rVert_{L_1}\le M\sum_{i=1}^\infty\lVert a_i\rVert_{L_q}\lVert y_i\rVert_{L_p},
\end{equation*}
and for all $x\in L_p$ for almost all $t\in E$ we have
\begin{equation*}
\bigl(Ax\bigr)(t)=\int_Ek(t,s)\,x(s)\,ds.
\end{equation*}
\end{theorem}

\begin{proof}
We denote by $A_i$ the operator $a_i\otimes y_i$, $i\in\mathbb N$. By Proposition~\ref{p:conj of L_p},
\begin{equation*}
\bigl(A_ix\bigr)(t)=y_i(t)\,\int_Ea_i(s)\,x(s)\,ds=\int_Ey_i(t)\,a_i(s)\,x(s)\,ds,\qquad t\in E.
\end{equation*}
The integral $\int_Ea_i(s)\,x(s)\,ds$ exists for all $x\in\mathcal L_p$ and is not changed when $a_i\in\mathcal L_q$ and $x\in\mathcal L_p$ are replaced by equivalent functions; the entire right-hand side is defined at the points $t$ at which the function $y_i\in\mathcal L_p$ is defined.

By Proposition~\ref{p:L_p<L_1}, the function
\begin{equation*}
k_i(t,s)=y_i(t)\,a_i(s)
\end{equation*}
belongs to $\mathcal L_1(E\times E,\mathbb C)$, with
\begin{align*}
\lVert k_i\rVert_{L_1}&=\lVert y_i(\cdot)a_i(\cdot\cdot)\rVert_{L_1}
=\lVert y_i(\cdot)\rVert_{L_1}\cdot\lVert a_i(\cdot\cdot)\rVert_{L_1}\\
&\le M^{\frac1{p}}\lVert a_i\rVert_{L_q}\cdot M^{\frac1{q}}\lVert y_i\rVert_{L_p}
=M\lVert a_i\rVert_{L_q}\cdot\lVert y_i\rVert_{L_p}.
\end{align*}
By estimate~\eqref{e:est:nuclear}, the series $\sum_{i=1}^\infty k_i$ converges absolutely in the semi-norm of the space $\mathcal L_1(E\times E,\mathbb C)$ to some function $k\in\mathcal L_1(E\times E,\mathbb C)$. By Proposition~\ref{p:a.e. conv}, the series $\sum_{i=1}^\infty k_i(t,s)$ converges absolutely to $k(t,s)$ for almost all $(t,s)\in E\times E$.

We denote by $F\subseteq E\times E$ the set of all points $(t,s)$ such that the series
\begin{equation*}
\sum_{i=1}^\infty k_i(t,s)=\sum_{i=1}^\infty y_i(t)\,a_i(s)
\end{equation*}
converges absolutely to $k(t,s)$. By the proved above, $F$ is a set of full measure, i.~e. $\lambda\bigl((E\times E)\setminus F\bigr)=0$. By Corollary~\ref{c:Fubini}, where exists a set $H\subseteq E$ of full measure such that the set $H_t=\{\,s\in E:\,(t,s)\in F\,\}$ is a set of full measure for all $t\in H$.

We take an arbitrary function $x\in\mathcal L_p(E,\mathbb C)$. We represent $A_ix$ in the form
\begin{equation*}
\bigl(A_ix\bigr)(t)=y_i(t)\int_Ea_i(s)\,x(s)\,ds.
\end{equation*}
By Proposition~\ref{p:Holder},
\begin{equation*}
\lVert A_ix\rVert_{L_p}\le\lVert y_i\rVert_{L_p}\cdot\lVert a_i\rVert_{L_q}\cdot\lVert x\rVert_{L_p}.
\end{equation*}
It is seen from this estimate and~\eqref{e:est:nuclear} that the series
\begin{equation*}
\sum_{i=1}^\infty A_ix=\sum_{i=1}^\infty y_i\int_Ea_i(s)\,x(s)\,ds
\end{equation*}
converges in the norm of $L_p$. By similar reasons, the series
\begin{equation*}
\sum_{i=1}^\infty |y_i|\int_E|a_i|(s)\,|x|(s)\,ds
\end{equation*}
also converges in the norm of $L_p$, where the symbol $|z|$ denotes the function $|z|(t)=\bigl|z(t)\bigr|$. From here and Proposition~\ref{p:a.e. conv}, it follows that
there exists a set $G$ of full measure such that the series
\begin{equation*}
\sum_{i=1}^\infty |y_i|(t)\int_E|a_i|(s)\,|x|(s)\,ds
=\sum_{i=1}^\infty |y_i|(t)\cdot\bigl\lVert |a_i|\cdot|x|\bigr\rVert_{L_1}
\end{equation*}
converges for all $t\in G$. This means that for $t\in G$ the series
\begin{equation}\label{e:sum yax}
\sum_{i=1}^\infty |y_i|(t)\cdot|a_i|\cdot|x|,
\end{equation}
which consists of the functions $|a_i|\cdot|x|\in\mathcal L_1(E,\mathbb C)$ (Proposition~\ref{p:Holder}) with the coefficients $|y_i|(t)$, converges in the $L_1$-norm. Therefore, for $t\in G$, the series
\begin{equation*}
\sum_{i=1}^\infty y_i(t)\cdot a_i\cdot x,
\end{equation*}
converges absolutely in the $L_1$-norm. By Proposition~\ref{p:a.e. conv}, for each $t\in G$, its sum coincides with the function
\begin{equation*}
s\mapsto\sum_{i=1}^\infty y_i(t)\;a_i(s)\cdot x(s)
\end{equation*}
for all $s\in G_t$, where $G_t\subset E$ is a set of full measure.

On the other hand, for all $t\in H$ and $s\in H_t$, we have
\begin{equation*}
\sum_{i=1}^\infty y_i(t)\;a_i(s)=k(t,s).
\end{equation*}
Consequently, for all $t\in H$ and $s\in H_t\cap D(x)$ (where $D(x)\subseteq E$ is the domain of $x$), we have
\begin{equation*}
\sum_{i=1}^\infty y_i(t)\;a_i(s)\cdot x(s)=k(t,s)\cdot x(s).
\end{equation*}
Then, a fortiori, for all $t\in G\cap H$ and $s\in H_t\cap D(x)$,
\begin{equation*}
\sum_{i=1}^\infty y_i(t)\;a_i(s)\cdot x(s)=k(t,s)\cdot x(s).
\end{equation*}

By the definition of nuclear operator, (for almost all $t$) we have
\begin{equation*}
\bigl(Ax\bigr)(t)=\sum_{i=1}^\infty \bigl(A_ix\bigr)(t)
=\sum_{i=1}^\infty y_i(t)\int_Ea_i(s)\,x(s)\,ds.
\end{equation*}
But for $t\in G\cap H$,
\begin{align*}
\sum_{i=1}^\infty y_i(t)\int_Ea_i(s)\,x(s)\,ds
&=\sum_{i=1}^\infty \int_E y_i(t)a_i(s)\,x(s)\,ds\\
&=\int_E\sum_{i=1}^\infty  y_i(t)a_i(s)\,x(s)\,ds\\
&=\int_E k(t,s)\,x(s)\,ds.\qed
\end{align*}
\renewcommand\qed{}
\end{proof}

\section{Locally nuclear operators}\label{s:locally nuclear:na}

\begin{definition}\label{def:loc nuk}
Let $g$ be a weight on $\mathbb Z^c$ satisfying assumptions {\rm(a)--(e)} from Definition~\ref{def:weight}.
Let $X$ be a Banach space.
We denote by $\mathoo s_{1,g}\bigl(\mathbb Z^c,\mathfrak S_1(X)\bigr)$ the set of all operators $T\in\mathoo B\bigl(l_p(\mathbb Z^c,X)\bigr)$ that can be represented in the form
\begin{equation}\label{e:operator D:2}
(Tx)_k=\sum\limits_{m\in\mathbb Z^c} b_{km}x_{k-m},\qquad k\in\mathbb Z^c,
\end{equation}
where $b_{km}\in\mathfrak S_1(X)$ satisfy the estimate
\begin{equation*}
\lVert b_{km}\rVert_{\mathfrak S_1(X)}\le\beta_{m}
\end{equation*}
for some $\beta\in l_{1,g}(\mathbb Z^c,\mathbb C)$, cf. Definition~\ref{def:s_{1,g}}.
By Proposition~\ref{p:l_{1,g} is an algebra}, $\mathoo s_{1,g}\bigl(\mathbb Z^c,\mathfrak S_1(X)\bigr)$ forms an algebra; this algebra has no unit if $X$ is infinite-dimensional (Corollary~\ref{c:nuclear is proper}). We call operators $T\in\mathoo s_{1,g}\bigl(\mathbb Z^c,\mathfrak S_1(X)\bigr)$ \emph{locally nuclear}.
\end{definition}

\begin{proposition}\label{p:K is an ideal in l_1}
The subalgebra $\mathoo s_{1,g}\bigl(\mathbb Z^c,\mathfrak S_1(X)\bigr)$ of locally nuclear operators forms an ideal in the algebra $\mathoo s_{1,g}\bigl(\mathbb Z^c,\mathoo B(X)\bigr)$.
\end{proposition}
\begin{proof}
As was noted above, $\mathoo s_{1,g}\bigl(\mathbb Z^c,\mathfrak S_1(X)\bigr)$ is a subalgebra of $\mathoo s_{1,g}\bigl(\mathbb Z^c,\mathoo B(X)\bigr)$.
Thus, it remains to show that $K\in\mathoo s_{1,g}\bigl(\mathbb Z^c,\mathfrak S_1(X)\bigr)$ and $T\in\mathoo s_{1,g}\bigl(\mathbb Z^c,\mathoo B(X)\bigr)$ imply $KT,TK\in\mathoo s_{1,g}\bigl(\mathbb Z^c,\mathfrak S_1(X)\bigr)$.

So, let $K\in\mathoo s_{1,g}\bigl(\mathbb Z^c,\mathfrak S_1(X)\bigr)$ and $T\in\mathoo s_{1,g}\bigl(\mathbb Z^c,\mathoo B(X)\bigr)$.
These mean that $K$ and $T$ admit the representations
\begin{equation*}
(Kx)_k=\sum\limits_{m\in\mathbb Z^c} a_{km}x_{k-m},\qquad
(Tx)_k=\sum\limits_{l\in\mathbb Z^c} b_{km}x_{k-m},\qquad k\in\mathbb Z^c,
\end{equation*}
where
\begin{equation*}
\lVert a_{km}\rVert_{\mathfrak S_1(X)}\le\alpha_{m},\qquad
\lVert b_{km}\rVert_{\mathoo B(X)}\le\beta_{m},
\end{equation*}
with $\alpha,\beta\in l_{1,g}(\mathbb Z^c,\mathbb C)$. By the definition of the product of operators, for any $x\in l_p(\mathbb Z^c,X)$, we have
\begin{equation*}
(KTx)_k=\sum\limits_{m\in\mathbb Z^c} a_{km}(Tx)_{k-m}
=\sum\limits_{m\in\mathbb Z^c} a_{km}
\sum\limits_{l\in\mathbb Z^c} b_{k-m,l}x_{k-m-l}
,\qquad k\in\mathbb Z^c.
\end{equation*}
Since $l_p(\mathbb Z^c,X)\subseteq l_\infty(\mathbb Z^c,X)$, the family $\{x_i:\;i\in\mathbb Z^c\}$ is bounded. Therefore the latter (double) series converges absolutely (for a fixed $k$). Consequently, any rearrangement of the series converges to the same sum.

We make the change $l=r-m$ in the internal sum:
\begin{equation*}
(KTx)_k=\sum\limits_{m\in\mathbb Z^c} a_{km}
\sum\limits_{r\in\mathbb Z^c} b_{k-m,r-m}x_{k-r}
,\qquad k\in\mathbb Z^c.
\end{equation*}
After that we interchange the order of summation:
\begin{equation}\label{e:KT}
(KTx)_k=\sum\limits_{r\in\mathbb Z^c}\Bigl(\sum\limits_{m\in\mathbb Z^c} a_{km}
 b_{k-m,r-m}\Bigr)x_{k-r}
,\qquad k\in\mathbb Z^c.
\end{equation}
By the estimate form Proposition~\ref{p:Pietch 6.3.2}, we have
\begin{equation*}
\lVert a_{km}b_{k-m,r-m}\rVert_{\mathfrak S_1(X)}\le
	\lVert a_{km}\rVert_{\mathfrak S_1(X)}\cdot\lVert b_{k-m,r-m}\rVert_{\mathoo B(X)}
\le\alpha_m\beta_{r-m}.
\end{equation*}
Therefore,
\begin{equation*}
\sum\limits_{m\in\mathbb Z^c} \lVert a_{km}b_{k-m,r-m}\rVert_{\mathfrak S_1(X)}
\le\sum\limits_{m\in\mathbb Z^c}\alpha_m\beta_{r-m}=(\alpha*\beta)_r.
\end{equation*}
The last estimate shows that the series $\sum\limits_{m\in\mathbb Z^c} a_{km} b_{k-m,r-m}$ converges absolutely in the norm $\lVert\cdot\rVert_{\mathfrak S_1}$. By the completeness of the ideal $\mathfrak S_1(X)$ (Proposition~\ref{p:Pietch 6.3.2}), this implies that the sum $\sum\limits_{m\in\mathbb Z^c} a_{km} b_{k-m,r-m}$ belongs to $\mathfrak S_1(X)$ and
\begin{equation*}
\Bigl\lVert \sum\limits_{m\in\mathbb Z^c} a_{km}b_{k-m,r-m}\Bigr\rVert_{\mathfrak S_1(X)}
\le(\alpha*\beta)_r.
\end{equation*}
By Proposition~\ref{p:l_{1,g} is an algebra}, $\alpha*\beta\in l_{1,g}(\mathbb Z^c,\mathbb C)$. Hence it follows from formula~\eqref{e:KT} that $KT\in\mathoo s_{1,g}\bigl(\mathbb Z^c,\mathfrak S_1(X)\bigr)$.

Similarly, one verifies that $TK\in\mathoo s_{1,g}\bigl(\mathbb Z^c,\mathfrak S_1(X)\bigr)$.
\end{proof}

We denote by $\widetilde{\mathoo s_{1,g}}\bigl(\mathbb Z^c,\mathfrak S_1(X)\bigr)$ the subalgebra derived from $\mathoo s_{1,g}\bigl(\mathbb Z^c,\mathfrak S_1(X)\bigr)$ by adjoining the unit element of the algebra $\mathoo B\bigl(l_p(\mathbb Z^c,X)\bigr)$ if $X$ is infinite-dimensional; and we denote by $\widetilde{\mathoo s_{1,g}}\bigl(\mathbb Z^c,\mathfrak S_1(X)\bigr)$ the algebra $\mathoo s_{1,g}\bigl(\mathbb Z^c,\mathfrak S_1(X)\bigr)$ itself provided that $X$ is finite-dimensional.

The following theorem is the main result of the present paper.
\begin{theorem}\label{t:fin:convolution dominayed:nuclear}
The subalgebra $\widetilde{\mathoo s_{1,g}}\bigl(\mathbb Z^c,\mathfrak S_1(X)\bigr)$ is full in the algebra $\mathoo B\bigl(l_p(\mathbb Z^c,X)\bigr)$ for all $1\le p\le\infty$.
\end{theorem}

\begin{proof}
Without loss of generality (Corollary~\ref{c:nuclear is proper}), we assume that $X$ is infinite-dimensional.

As a first step, we show that the subalgebra $\widetilde{\mathoo s_{1,g}}\bigl(\mathbb Z^c,\mathfrak S_1(X)\bigr)$ is full in the algebra $\mathoo s_{1,g}\bigl(\mathbb Z^c,\,\mathoo B(X)\bigr)$.
Let an operator $\lambda\mathbf1+T$, where $T\in\mathoo s_{1,g}\bigl(\mathbb Z^c,\mathfrak S_1(X)\bigr)$, be invertible in $\mathoo s_{1,g}\bigl(\mathbb Z^c,\,\mathoo B(X)\bigr)$. We consider (Proposition~\ref{p:K is an ideal in l_1}) the quotient morphism of algebras
\begin{equation*}
\varphi:\,\mathoo s_{1,g}\bigl(\mathbb Z^c,\,\mathoo B(X)\bigr)\to\mathoo s_{1,g}\bigl(\mathbb Z^c,\,\mathoo B(X)\bigr)/\mathoo s_{1,g}\bigl(\mathbb Z^c,\mathfrak S_1(X)\bigr).
\end{equation*}
By the definition of quotient morphism, we have
\begin{equation*}
\varphi(\lambda\mathbf1_{\mathoo s_{1,g}(\mathbb Z^c,\,\mathoo B(X))}+T)=\lambda\mathbf1_{\mathoo s_{1,g}(\mathbb Z^c,\,\mathoo B(X))/\mathoo s_{1,g}(\mathbb Z^c,\mathfrak S_1(X))}.
\end{equation*}
The element $\lambda\mathbf1_{\mathoo s_{1,g}(\mathbb Z^c,\,\mathoo B(X))/\mathoo s_{1,g}(\mathbb Z^c,\mathfrak S_1(X))}$ is invertible by
Proposition~\ref{p:morphism of algebras}, since the element $\lambda\mathbf1+T$ is invertible. Therefore, $\lambda\neq0$. Consequently,
\begin{align*}
\bigl(\varphi(\lambda\mathbf1+T)\bigr)^{-1}&=\varphi\bigl((\lambda\mathbf1+T)^{-1}\bigr)=(\lambda\mathbf1_{\mathoo s_{1,g}(\mathbb Z^c,\,\mathoo B(X))/\mathoo s_{1,g}(\mathbb Z^c,\mathfrak S_1(X))})^{-1}\\
&=\frac1\lambda\mathbf1_{\mathoo s_{1,g}(\mathbb Z^c,\,\mathoo B(X))/\mathoo s_{1,g}(\mathbb Z^c,\mathfrak S_1(X))}.
\end{align*}
This equality implies that $(\lambda\mathbf1+T)^{-1}=\frac1\lambda\mathbf1+T_1$, where $T_1\in\mathoo s_{1,g}\bigl(\mathbb Z^c,\mathfrak S_1(X)\bigr)$, which means that $(\lambda\mathbf1+T)^{-1}\in\widetilde{\mathoo s_{1,g}}\bigl(\mathbb Z^c,\mathfrak S_1(X)\bigr)$.
	
To complete the proof, it is sufficient to recall that the subalgebra $\mathoo s_{1,g}\bigl(\mathbb Z^c,\mathoo B(X)\bigr)$ is full in the algebra $\mathoo B\bigl(l_p(\mathbb Z^c,X)\bigr)$ by Theorem~\ref{t:s_g is full in s}.
\end{proof}

\section{Locally nuclear operators in the spaces $L_p$}\label{s:locally nuclear:na:L_p}
We represent the set $\mathbb R^c$ as the disjoint union (i.~e. a union of disjoint sets)
\begin{equation*}
\mathbb R^c=\bigsqcup_{m\in\mathbb Z^c}[0,1)^c+m,
\end{equation*}
where
\begin{equation*}
[0,1)^c+m=\bigsqcup_{m=(m_1,m_2,\dots,m_c)\in\mathbb Z^c}
[m_1,m_1+1)\times[m_2,m_2+1)\times\dots\times[m_c,m_c+1)
\end{equation*}
and $m=(m_1,m_2,\dots,m_c)$.

\begin{proposition}\label{p:R c via Z c}
The following properties hold.
\begin{enumerate}
 \item[{\rm(a)}] A set $E\subseteq\mathbb R^c$ is measurable if and only if its intersection with each of the sets $[0,1)^c+m$, $m\in\mathbb Z^c$, is summable.
 \item[{\rm(b)}] A set $N\subseteq\mathbb R^c$ is negligible if and only if its intersection with each of the sets $[0,1)^c+m$, $m\in\mathbb Z^c$, is negligible.
 \item[{\rm(c)}] A function $x:\,\mathbb R^c\to\mathbb C$ is measurable if and only if its restriction to each of the sets $[0,1)^c+m$, $m\in\mathbb Z^c$, is measurable.
 \item[{\rm(d)}] A function $x:\,\mathbb R^c\to\mathbb C$ is negligible if and only if its restriction to each of the sets $[0,1)^c+m$, $m\in\mathbb Z^c$, is negligible.
\end{enumerate}
\end{proposition}
\begin{proof}
The proof is straightforward.
\end{proof}

\begin{proposition}[{\rm\cite{Fournier-Stewart85}, \cite[1.6.3]{Kurbatov99}}]\label{p:L_p viz l_p}
Let $1\le p\le\infty$. Then the mapping $\varphi:\,\mathcal L_p(\mathbb R^c,\mathbb C)\to l_p\bigl(\mathbb Z^c,\mathcal L_p([0,1)^c,\mathbb C)\bigr)$ given by the rule $\varphi(x)=\{x_m\}$, where
\begin{equation*}
x_m(t)=x(t+m),\qquad t\in[0,1)^c,
\end{equation*}
defines {\rm(}after identifying of equivalent functions{\rm)} an isometric isomorphism $\varphi:\,L_p(\mathbb R^c,\mathbb C)\to l_p\bigl(\mathbb Z^c,L_p([0,1)^c,\mathbb C)\bigr)$ {\rm(}which we denote by the same symbol $\varphi${\rm)}.
\end{proposition}

\begin{proof}
So, we consider the mapping $\varphi:\,\mathcal L_p(\mathbb R^c,\mathbb C)\to l_p\bigl(\mathbb Z^c,\mathcal L_p([0,1)^c,\mathbb C)\bigr)$ defined by the rule $\varphi(x)=\{x_m\}$, where
\begin{equation*}
x_m(t)=x(t+m),\qquad t\in[0,1)^c.
\end{equation*}
Clearly, for any $x\in\mathcal L_p(\mathbb R^c,\mathbb C)$, the sequence $\{x_m\}$ consists of measurable functions and
\begin{equation*}
\lVert x\rVert_{L_p}=\bigl\lVert \{\lVert x_m\rVert_{L_p}\}\bigr\rVert_{l_p}
\end{equation*}
or, in more detail,
\begin{align*}
\lVert x\rVert_{L_p}=\sqrt[p]{\int_{\mathbb R^c}|x(t)|^p\,dt}
&=\sqrt[p]{\sum_{m\in\mathbb Z^c}\int_{[0,1)^c}|x_m(t)|^p\,dt},&& p<\infty,\\
\lVert x\rVert_{L_\infty}=\esssup_{t\in\mathbb R^c}|x(t)|&=\sup_{m\in\mathbb Z^c}\esssup_{t\in[0,1)^c}|x_m(t)|,&& p=\infty.
\end{align*}
Thus, $\varphi$ in fact acts from $\mathcal L_p(\mathbb R^c,\mathbb C)$ to $l_p\bigl(\mathbb Z^c,\mathcal L_p([0,1)^c,\mathbb C)\bigr)$ and preserves the norm.

The linearity of $\varphi$ is evident. The preservation of the norm implies that $\varphi$ is injective.

Let $\{x_m\}\in l_p\bigl(\mathbb Z^c,\mathcal L_p([0,1)^c,\mathbb C)\bigr)$. Obviously, the sequence $\{x_m\}$ is the inverse image of the function
\begin{equation*}
x(t)=x_m(t-m)\qquad\text{when } t\in [0,1)^c+m.
\end{equation*}
Thus, $\varphi$ is surjective.

By Proposition~\ref{p:R c via Z c}(d), a measurable function $x$ is negligible if and only if all members of the sequence $\varphi(x)=\{x_m\}$ are negligible functions.
Therefore, $\varphi$ generates an isomorphic isomorphism $\varphi:\,L_p(\mathbb R^c,\mathbb C)\to l_p\bigl(\mathbb Z^c,L_p([0,1)^c,\mathbb C)\bigr)$.
\end{proof}

\begin{definition}\label{def:loc nuk:L}
Since the spaces $L_p(\mathbb R^c,\mathbb C)$ and $l_p\bigl(\mathbb Z^c,L_p([0,1)^c,\mathbb C)\bigr)$ are isomorphic, the algebras of operators $\mathoo B\bigl(L_p(\mathbb R^c,\mathbb C)\bigr)$ and $\mathoo B\bigl(l_p\bigl(\mathbb Z^c,L_p([0,1)^c,\mathbb C)\bigr)\bigr)$ are isomorphic as well.
We denote by $\mathoo S_{1,g}(\mathbb R^c,\mathfrak S_1)=\mathoo S_{1,g}\bigl(\mathbb R^c,\mathfrak S_1\bigl(L_p([0,1)^c,\mathbb C)\bigr)\bigr)$, $1\le p\le\infty$,
the set of all operators $A\in\mathoo B\bigl(L_p(\mathbb R^c,\mathbb C)\bigr)$ that correspond to operators of the class $\mathoo s_{1,g}\bigl(\mathbb Z^c,\mathfrak S_1\bigl(L_p([0,1)^c,\mathbb C)\bigr)\bigr)$ according to the isomorphism $\varphi$ described in Proposi\-tion~\ref{p:L_p viz l_p}. More precisely, an operator $A$ belongs to the class $\mathoo S_{1,g}(\mathbb R^c,\mathfrak S_1)$ if and only if the operator $T=\varphi A\varphi^{-1}$, which renders commutative the diagram
\begin{equation}\label{e:CD}
\begin{CD}
L_p @>\varphi>> l_p\\
@VV{A}V @VV{T}V\\
L_p @>\varphi>> l_p,
\end{CD}
\end{equation}
belongs to the class $\mathoo s_{1,g}\bigl(\mathbb Z^c,\mathfrak S_1\bigl(L_p([0,1)^c,\mathbb C)\bigr)\bigr)$. We call operators belonging to the class $\mathoo S_{1,g}(\mathbb R^c,\mathfrak S_1)$ \emph{locally nuclear} as well.
\end{definition}

We denote by $\widetilde{\mathoo S_{1,g}}(\mathbb R^c,\mathfrak S_1)$ the subalgebra derived from $\mathoo S_{1,g}(\mathbb R^c,\mathfrak S_1)$ by adjoining the unit element of the algebra $\mathoo B\bigl(L_p(\mathbb R^c,X)\bigr)$.

The following theorem is the most interesting special case of Theorem~\ref{t:fin:convolution dominayed:nuclear}.

\begin{theorem}\label{t:convolution dominayed:nuclear:2}
The subalgebra $\widetilde{\mathoo S_{1,g}}(\mathbb R^c,\mathfrak S_1)$ is full in the algebra $\mathoo B\bigl(L_p(\mathbb R^c,X)\bigr)$ for all $1\le p\le\infty$.
\end{theorem}

\begin{proof}
The proof follows from Theorem~\ref{t:fin:convolution dominayed:nuclear}.
\end{proof}

\begin{theorem}\label{t:loc nuclear oper in L_p}
Let $1\le p<\infty$. Then, for each operator $A\in\mathoo S_{1,g}(\mathbb R^c,\mathfrak S_1)$, there exists a measurable function
\begin{equation*}
n:\,\mathbb R^c\times \mathbb R^c\to\mathbb C
\end{equation*}
such that for any $x\in L_p(\mathbb R^c,\mathbb C)$ at almost all points $t\in\mathbb R^c$ {\rm(}the following integral exists and{\rm)}
\begin{equation*}
\bigl(Ax\bigr)(t)=\int_{\mathbb R^c}n(t,s)\,x(s)\,ds.
\end{equation*}
\end{theorem}

\begin{proof}
We consider the operator $T=\varphi A\varphi^{-1}$ rendering commutative diagram~\eqref{e:CD}.
By definition, $T\in\mathoo s_{1,g}\bigl(\mathbb Z^c,\mathfrak S_1\bigl(L_p([0,1)^c,\mathbb C)\bigr)\bigr)$. We restore the operator $A$ by means of the formula $A=\varphi^{-1} T\varphi$.
In accordance with Proposition~\ref{p:L_p viz l_p}, we assign to the function $x$ the family $\{x_m\}=\varphi(x)$:
\begin{equation*}
x_m(t)=x(t+m),\qquad t\in[0,1)^c,\;m\in\mathbb Z^c.
\end{equation*}
We denote briefly the family $\{x_m\}=\varphi(x)$ by $z$.
We apply the operator $T$ to $z$. By virtue of~\eqref{e:operator D:2},
\begin{equation*}
(Tz)_k=\sum\limits_{m\in\mathbb Z^c} b_{km}x_{k-m},\qquad k\in\mathbb Z^c,
\end{equation*}
where $b_{km}\in\mathfrak S_1\bigl(L_p([0,1)^c,\mathbb C)\bigr)$,
and the series converges absolutely; more precisely,
\begin{equation*}
\|b_{km}\|\le\lVert b_{km}\rVert_{\mathfrak S_1}\le\beta_{m}.
\end{equation*}

By Proposition~\ref{p:a.e. conv}, for almost all $t\in[0,1)^c$, we have
\begin{equation*}
\bigl((Tz)_k\bigr)(t)=\sum\limits_{m\in\mathbb Z^c} \bigl(b_{km}x_{k-m}\bigr)(t),\qquad k\in\mathbb Z^c.
\end{equation*}
Applying to the sequence $\{(Tz)_i\}$ the isomorphism $\varphi^{-1}$, we arrive at the function $Ax$. According to the previous formula,
\begin{align*}
\bigl(Ax\bigr)(t)&=\bigl(\varphi^{-1}\bigl(\{(Tz)_i\}\bigr)\bigr)(t)=\bigl[(Tz)_k\bigr](t-k)\\
&=\sum\limits_{m\in\mathbb Z^c}\bigl(b_{km}x_{k-m}\bigr)(t-k),\qquad t\in[0,1)^c+k.
\end{align*}
We consider the blocks $b_{km}\in\mathoo B\bigl(L_p([0,1)^c,\mathbb C)\bigr)$, $k,m\in\mathbb Z^c$, which constitute the operator $T$. By Theorem~\ref{t:nuclear oper in L_p}, for each of them there exists a measurable function $n_{km}:\,[0,1)^c\times[0,1)^c\to\mathbb C$ such that for any $u\in L_p([0,1)^c,\mathbb C)$ at almost all  $t\in[0,1)^c$ {\rm(}the following integral exists and{\rm)}
\begin{equation*}
\bigl(b_{km}u\bigr)(t)=\int_{[0,1)^c}n_{km}(t,s)\,u(s)\,ds.
\end{equation*}
In particular, we have (almost everywhere)
\begin{equation*}
\bigl(b_{km}x_{k-m}\bigr)(t-k)=\int_{[0,1)^c}n_{km}(t-k,s)\,x_{k-m}(s)\,ds,\qquad t\in[0,1)^c+k,
\end{equation*}
or
\begin{align*}
\bigl(b_{km}x_{k-m}\bigr)(t-k)&=\int_{[0,1)^c+k-m}n_{km}(t-k,\sigma-k+m)\,x_{k-m}(\sigma-k+m)\,d\sigma\\
&=\int_{[0,1)^c+k-m}n_{km}(t-k,\sigma-k+m)\,x(\sigma)\,d\sigma,\qquad t\in[0,1)^c+k.
\end{align*}
Hence, for almost all $t\in[0,1)^c+k$ (which implies that for almost all $t\in\mathbb R^c$)
\begin{align*}
\bigl(Ax\bigr)(t)&=\sum\limits_{m\in\mathbb Z^c}\bigl(b_{km}x_{k-m}\bigr)(t-k)\\
&=\sum\limits_{m\in\mathbb Z^c}\int_{[0,1)^c+k-m}n_{km}(t-k,s-k+m)\,x(s)\,ds\\
&=\int_{\mathbb R^c}n(t,s)\,x(s)\,ds,
\end{align*}
where
\begin{equation*}
n(t,s)=n_{km}(t-k,s-k+m),\qquad t\in[0,1)^c+k,\;s\in[0,1)^c+k-m,
\end{equation*}
or
\begin{equation*}
n(t,s)=n_{km}(t-k,s-l),\qquad t\in[0,1)^c+k,\;s\in[0,1)^c+l.\qed
\end{equation*}
\renewcommand\qed{}
\end{proof}

\providecommand{\bysame}{\leavevmode\hbox to3em{\hrulefill}\thinspace}
\providecommand{\MR}{\relax\ifhmode\unskip\space\fi MR }
\providecommand{\MRhref}[2]{%
  \href{http://www.ams.org/mathscinet-getitem?mr=#1}{#2}
}
\providecommand{\href}[2]{#2}

\end{document}